%% file: ms.tex
\documentclass[11pt]{amsart}
\newcommand\mtop{1in}
\newcommand\mbottom{1in}
\newcommand\mleft{1.2in}
\newcommand\mright{1.2in}
\input{usepackage.tex}

\usepackage{ifthen}
\usepackage{pifont}
\usepackage[new]{old-arrows}
\usepackage[nolabel]{showlabels}

\begin{document}

\newcommand{\theoremnumstyle}{section}
\input{preamble8.tex}
\input{newtheorem2.tex}
\input{newcommands-thesis.tex}
\renewcommand{\showlabelfont}{\tiny\slshape\color{mgreen}}

\author{Eva Belmont}
\title{A Cartan-Eilenberg spectral sequence for a non-normal extension}
\lhead{}\chead{\it\small A Cartan-Eilenberg spectral sequence
for a non-normal extension}\rhead{}\lfoot{}\cfoot{\thepage}\rfoot{}

\maketitle

\begin{abstract}
Let $\Phi\to \Gamma\to \Sigma$ be a conormal extension of Hopf algebras over a
commutative ring $k$, and let
$M$ be a $\Gamma$-comodule. The Cartan-Eilenberg spectral sequence $$ E_2 =
\mathrm{Ext}_\Phi(k,\mathrm{Ext}_\Sigma(k,M)) \implies \mathrm{Ext}_\Gamma(k,M)$$ is a standard tool for
computing the Hopf algebra cohomology of $\Gamma$ with coefficients in $M$ in
terms of the cohomology of the pieces $\Phi$ and $\Sigma$. Bruner and Rognes,
generalizing a construction of Davis and Mahowald, have introduced a
generalization of the Cartan-Eilenberg spectral sequence converging to
$\mathrm{Ext}_\Gamma(k,M)$ that can be defined when $\Phi =
\Gamma\cotensor_\Sigma k$ is compatibly an algebra and a $\Gamma$-comodule. We
offer a concrete cobar-like construction that fits into their framework, and
show how this work fits into a larger story. In particular, we show that this
spectral sequence is isomorphic, starting at the $E_1$ page, to both the Adams
spectral sequence in the stable category of $\Gamma$-comodules as studied by
Margolis and Palmieri, and to a filtration spectral sequence on the cobar
complex for $\Gamma$ originally due to Adams. We obtain a description of the
$E_2$ term under an additional flatness assumption. We discuss applications to
computing localizations of the Adams spectral sequence $E_2$ page.
\end{abstract}

\section{Introduction}
Suppose $\Gamma$ is a Hopf algebra over a commutative ring $k$ and we wish to
calculate its Hopf algebra cohomology $\Ext_\Gamma(k,k)$. If $k=\F_p$ and
$\Gamma$ is a group ring $\F_p[G]$, then this cohomology is by definition the
mod-$p$ group cohomology $H^*(G,\F_p)$, and a short exact sequence of groups $1\to
N\to G\to G/N\to 1$ gives rise to a Lyndon-Hochschild-Serre spectral sequence
$$ E_2^{**} = H^*(G/N,H^*(N,\F_p)) \implies H^*(G,\F_p) $$
(equivalently $\Ext^*_{\F_p[G/N]}(\F_p,\Ext^*_{\F_p[N]}(\F_p,\F_p))\implies
\Ext^*_{\F_p[G]}(\F_p,\F_p)$). The analogue for more general $\Gamma$ is called the
Cartan-Eilenberg spectral sequence (alternatively the extension spectral
sequence or the change-of-rings spectral sequence); there are a number of
variants that are defined in various settings.

In the setting described in Cartan and Eilenberg's classic book (see
\cite[XVI.5(2)$_4$]{cartan-eilenberg}), one begins with an algebra map
$\Gamma\to \Sigma$ and a $\Gamma$-module $M$. Then the
composite functor spectral sequence associated to the functors
$\Hom_\Sigma(k, -)\circ \Hom_\Gamma(\Sigma, -)$ has the form
$$ E_2 = \Ext_\Sigma(k,\Ext_\Gamma(\Sigma,M))\implies \Ext_\Gamma(k,M). $$
Given a normal algebra map $i:\Phi\to \Gamma$ (i.e., the left ideal $\Gamma\cdot
i(\Phi)$ is also a right ideal), one can define a quotient $\Gamma\tensor_\Phi
k$ that is an algebra. If $\Sigma$ can be expressed as such a quotient where
$\Gamma$ is projective as a $\Phi$-module, then one may apply the change of
rings theorem to obtain the more familiar form \cite[Theorem
XVI.6.1]{cartan-eilenberg}
\begin{equation} \label{cess-model}E_2 = \Ext_\Sigma(k,\Ext_\Phi(k,M))\implies
\Ext_\Gamma(k,M). \end{equation}
This reduces to the group cohomology example above in the case when the sequence
$\Phi\to \Gamma\to \Sigma$ is $\F_p[N]\to \F_p[G]\to \F_p[G/N]$. Normality of
$i$, which guarantees the quotient $\Sigma$ is an algebra, is the analogue of
normality of the subgroup $N$.

The case where $\Gamma$ is a Hopf algebra (or Hopf algebroid)
is of particular interest to stable homotopy theory.
A detailed account of the Cartan-Eilenberg spectral sequence
in the setting of Hopf algebroids can be found in \cite[Appendix
A1.3]{green}. The basic setup involves a normal extension of Hopf algebras,
i.e., a sequence of Hopf algebra maps $\Phi\intoo{i} \Gamma\ontoo{\pi} \Sigma:=
\Gamma\tensor_\Phi k$ such that $i$ is a normal map of $k$-algebras, and $i$ and
$\pi$ split over $k$. The $E_2$ page has the same form as \eqref{cess-model}.
One motivating example is the extension generated by the quotient $A\to
A/\beta$ of the Steenrod algebra obtained by modding out by the action of the
Bockstein. Working in the Hopf algebra setting with more general coefficient
comodules, Singer \cite[Chapter 4]{singer-book} identifies the Cartan-Eilenberg
spectral sequence as the spectral sequence of a bi-cosimplicial
commutative algebra and describes a general theory of power operations acting
on it.

In this paper we assume weaker hypotheses: in particular, $\Sigma$ need not be
an algebra, in which case the classical $E_2$ page and the first of the
composite functors is not even defined. Davis and Mahowald
\cite{davis-mahowald-A(2)} were the first to develop a Cartan-Eilenberg type
spectral sequence when they studied $\Ext_{A(2)}(M, \F_2)$ for $A(2)$-modules
$M$ using the non-normal map $A(1)\to A(2)$ of Hopf algebras. In this paper we
study a generalization of their construction due to Bruner and Rognes
\cite{bruner-rognes}.

As our main application is to the localized cohomology of the dual Steenrod
algebra, where it is more convenient to work with comodules than modules, the
rest of this discussion will pertain to the dual setting to that described
above. Henceforth, $\Ext_\Gamma$ will denote comodule Ext---that is, derived
functors of Hom in the category of comodules over a Hopf algebra $\Gamma$. The
classical Cartan-Eilenberg spectral sequence in this setting has the form $$ E_2
= \Ext_\Phi(k,\Ext_\Sigma(k,M))\implies \Ext_\Gamma(k,M) $$ where $M$ is a
$\Phi$-comodule and $\Phi\to \Gamma\to \Sigma$ is a conormal extension of Hopf
algebras (i.e., an extension such that $\Gamma\cotensor_\Sigma k =
k\cotensor_\Sigma \Gamma$ as sub-vector spaces).


Suppose $\Gamma\to \Sigma$ is a surjection of Hopf algebras, and the map
$\Phi\colonequals \Gamma\cotensor_\Sigma k \to \Gamma$ is a map of
$\Gamma$-comodule algebras. For a $\Gamma$-comodule $M$, we consider 
three spectral sequences converging to $\Ext_\Gamma^*(k,M)$.
\begin{enumerate} 
\item Given a $\Gamma$-comodule resolution of $k$ of the form
$\Phi\tensor R^0\to \Phi\tensor R^1 \to \dots$ for $\Gamma$-comodules $R^n$, Bruner and Rognes describe a
Cartan-Eilenberg type spectral sequence converging to $\Ext_\Gamma^*(k,M)$.%
\footnote{The spectral sequence we actually discuss computes
$\Cotor_\Gamma(M,N)$, not $\Ext_\Gamma(M,N)$. We gloss over this difference
because these are naturally isomorphic when $M=k$, which is our main case of
interest. An analogous construction would work for comodule Ext in general,
however.}
While Bruner and Rognes are primarily interested in the case where $\Sigma$ is
an exterior algebra and the complex $\Phi\tensor R^*$ is a minimal resolution,
we construct a resolution $R^n = \Phi^{\tensor n}$ which, in the case $\Phi$ is
a coalgebra, reduces to the $\Phi$-cobar resolution of $M$. Our construction
just depends on a $\Gamma$-comodule map $\Phi\to \Gamma$ where $\Phi$ is a
$\Gamma$-comodule algebra; the presence of $\Sigma$ such that $\Phi =
\Gamma\cotensor_\Sigma k$ simply gives a more convenient form for the $E_1$ page.

\item Margolis \cite{margolis} and Palmieri \cite{palmieri-book} have studied
the generalized Adams spectral sequence constructed in the category of stable
$\Gamma$-comodules, a close cousin of the derived category of
$\Gamma$-comodules. If $\Phi$ is a $\Gamma$-comodule algebra and $M$ is a
$\Gamma$-comodule, then the $\Phi$-based Adams spectral sequence in
$\Stable(\Gamma)$ for $M$ converges to $\Ext_\Gamma(k,M)$.

\item The third construction is a filtration spectral sequence on the cobar
complex on $\Gamma$ due to Adams \cite{adams-hopf-invariant}. Though originally
studied in the case where $\Phi\to \Gamma\to \Sigma$ is an extension of Hopf
algebras, the filtration spectral sequence
itself may be defined in the setting here (where $\Phi$ is a $\Gamma$-comodule
algebra) without
modification. 
\end{enumerate}
The main results of this paper can be summarized as follows.
\begin{theorem}\label{thm:equivalence}
The spectral sequences (1), (2), and (3) coincide at the $E_1$ page, which has
the form
$$ E_1^{s,t} = \Ext^t_\Sigma(k,\bar{\Phi}^{\tensor s}\tensor M) $$
where $\bar{\Phi}$ denotes the coaugmentation ideal.
\end{theorem}
These three constructions have different advantages: since (1) is the spectral
sequence associated to a bi-cosimplicial commutative algebra, it has power
operations due to the general theory of Sawka \cite{sawka}. The Adams spectral
sequence presentation (2) naturally comes with an $E_2$ term under an analogue
of the classical Adams flatness condition (see Corollary \ref{cor:E_2} below). The filtration spectral sequence
presentation (3) is convenient for explicit computations in low degrees. Our
comparison theorems enable one to use all of these desirable properties without
regard to a choice of underlying model.

\begin{corollary}\label{cor:E_2}
If $\Ext^*_\Sigma(k,\Phi)$ is flat as a module over $\Ext_\Sigma^*(k,k)$, then
the spectral sequences (1), (2), and (3) have $E_2$ page
$$ E_2^{**} \isom \Ext^*_{\Ext_\Sigma(k,\Phi)}(\Ext_\Sigma(k,k),
\Ext_\Sigma(k,M)). $$
\end{corollary}
In Section \ref{section:applications} we give an example of a setting in which
the new $E_2$ page is defined but the classical Cartan-Eilenberg spectral
sequence is not; in general, we expect such settings to involve computing
localized Ext groups.
Given a non-nilpotent element $x\in \Ext_\Gamma(k,k)$ whose image in
$\Ext_\Sigma(k,k)$ is non-nilpotent, one may localize the entire spectral
sequence construction, obtaining a spectral sequence that, in good cases,
converges to $x^{-1} \Ext_\Gamma(k,M)$ (though we note that convergence must be
checked separately). In many cases of interest, the $E_2$ condition holds only
after inverting $x$, in which case the $E_2$ page of the localized spectral
sequence has the form
$$ E_2^{**} \isom \Ext^*_{x^{-1}\Ext_\Sigma(k,\Phi)} (x^{-1}\Ext_\Sigma(k,k),\
x^{-1}\Ext_\Sigma(k,M)).$$
(The idea is that the flatness condition may hold for the
$x$-local part of $\Ext_\Sigma(k,\Phi)$ but not the torsion part.)

Localized $\Ext$ groups have been studied in many contexts, often because the
localization represents the more tractable part of an otherwise complicated
$\Ext$ group of interest.
For example, Davis and Mahowald's work on $v_1$-local Ext groups over the
Steenrod algebra \cite{davis-mahowald-v_1} is an important part of understanding
the $E_2$ page of the $bo$-based Adams spectral sequence (also see
\cite{bo-Adams} for a detailed study of the relationship between the
$v_1$-periodic and $v_1$-torsion parts). More generally, for a type $n$ spectrum
$X$, the $v_n$-localized Adams $E_2$ page for $X$ is the algebraic analogue of
the chromatic localization $\pi_*(v_n^{-1} X)$. In Section
\ref{section:applications} we use the techniques of this paper to give a new
proof of May and Milgram's calculation of the $p$-towers in the Adams spectral
sequence $E_2$ term for a finite spectrum; this calculates the $E_2$ page above
a line of slope ${1\over 2p-2}$. This is a much easier analogue of the author's
study \cite{b10-paper} of a different localization of the Adams $E_2$ page for
the sphere at $p=3$.

\subsection*{Outline} We begin in Section \ref{section:applications} by
discussing the motivating application for this work, a localization of the Adams
$E_2$ page for the sphere. In Section \ref{section:CESS}, we review the
classical construction of the Cartan-Eilenberg spectral sequence for an
extension of Hopf algebras $\Phi\to \Gamma \to \Sigma$, and define a variation
(the spectral sequence mentioned in (1)) that makes sense when $\Phi$ is only a
$\Gamma$-comodule algebra. The main step is to replace the $\Phi$-cobar
resolution of a $\Phi$-comodule $M$ (which does not make sense when $\Phi$ does
not have a coalgebra structure) with a $\Gamma$-comodule resolution. This
amounts to describing a specific resolution $\Phi\tensor R^*$ for use in Bruner
and Rognes' construction.

In Section \ref{section:MPASS}, we review Margolis and
Palmieri's Adams spectral sequence, and prove that the spectral sequence (1)
coincides with this one at $E_1$. This extends a remark of Palmieri
\cite[\S1.4]{palmieri-book}, who notes that the spectral sequence he studies
coincides with the Cartan-Eilenberg spectral sequence in the case that the
extension is conormal (the coalgebra analogue of the normality condition
mentioned above).

Section \ref{section:CESS-vs-filtration} is devoted to comparing the spectral
sequences (1) and (3). Adams \cite[\S2.3]{adams-hopf-invariant} mentions
(without proof) that (3) coincides with the classical Cartan-Eilenberg spectral
sequence in the case when the latter is defined. A proof of this fact is given
in \cite[A1.3.16]{green}, attributed to Ossa. Our comparison proof is based on
Ossa's. This involves the use of explicit formulas for the
iterated shear isomorphism and its inverse, which are established in the
appendix.

\subsection*{Acknowledgements} The work in this paper represents part of my
thesis work. I am greatly indebted to Haynes Miller, my graduate
advisor, for his guidance and support throughout this project, and more
specifically for an early version of the construction in Section
\ref{subsection:new-CESS}. I'd also like to thank Paul Goerss for informing me
of the existence of \cite{davis-mahowald-A(2)} and \cite{bruner-rognes}, and
Bruner and Rognes for providing me with a draft of their work on this subject.

\section{Application: Localized Adams $E_2$ page}\label{section:applications}
The motivation for the work in this paper was an attempt to calculate a
localization of $\Ext_P(\F_3,\F_3)$, where $P=\F_3[\xi_1,\xi_2,\dots]\subset A$
is the dual algebra of reduced powers, using a Cartan-Eilenberg spectral
sequence based on the extension
$$ B:=\F_3[\xi_1^3,\xi_2,\xi_3,\dots] \to P \to \F_3[\xi_1]/\xi_1^3. $$
Since $B$ is not a sub-coalgebra of $P$ (e.g. $\Delta(\xi_n)=\xi_{n-1}^3\tensor
\xi_1 + \dots$ is not contained in $P\tensor P$), the standard Cartan-Eilenberg
spectral sequence is not defined; however, $B$ is a $P$-comodule algebra, and so
the generalized Cartan-Eilenberg spectral sequence described in this paper can
be used. Furthermore, while the Adams flatness condition in Corollary
\ref{cor:E_2} does not hold for this extension, it does hold after
inverting the polynomial class $b_{10}\in
\Ext^2_{\F_3[\xi_1]/(\xi_1^3)}(\F_3,\F_3)$. The computation of the resulting
localized spectral sequence is the
subject of \cite{b10-paper} which, at various points, crucially uses each of
the three forms of the spectral sequence discussed in this paper.

In the rest of this section we give a complete calculation of a simpler
version of this problem, which recovers a classical result on the
rationalization of the sphere as a straightforward consequence of the techniques
in this paper. Recall that the dual Steenrod algebra at $p>2$
has the form $A = \F_p[\xi_1,\xi_2,\dots]\tensor E[\tau_0,\tau_1,\dots]$, and 
let $a_0  = [\tau_0] \in \Ext^1_A(\F_p,\F_p)$. This class survives the Adams
spectral sequence $\Ext_A(\F_p,\F_p)\implies \pi_*S^\hhat_p$ and is detected in
homotopy by the multiplication-by-$p$ map $(S\too{p} S)\in \pi_0S$. The following
result, which should be seen as the algebraic version of Serre's calculation
\cite{serre-classes} of the homotopy of the rationalized sphere,
is originally due to May and Milgram \cite{may-milgram}.
\begin{proposition}\label{adams-towers}
Let $M$ be an $A$-comodule. For $p>2$,
$$a_0^{-1} \Ext_A(\F_p,M) = a_0^{-1}\Ext_{E[\tau_0]}(\F_p,M).$$
In particular, the $a_0$-localized Adams $E_2$ page for the sphere is $a_0^{-1}
\Ext_A(\F_p,\F_p) = \F_p[a_0^{\pm 1}]$.
\end{proposition}
\begin{proof}
Let $E[x]$ denote the exterior algebra over $\F_p$ on a class $x$. The sequence of
algebras
$$ C:= \F_p[\xi_1,\xi_2,\dots]\tensor E[\tau_1,\dots] \to A\to E[\tau_0]  $$
is not an extension of Hopf algebras, as the diagonal on $\tau_n\in C$ does not
lie in $C\tensor C$. However, $C = A\cotensor_{E[\tau_0]}\F_p$ is an $A$-comodule algebra, and so the
spectral sequences of Theorem \ref{thm:equivalence} are defined. In particular,
we have an $E_1$ page $$ E_1^{s,t} = \Ext_{E[\tau_0]}^t(\F_p, \bar{C}^{\tensor s}\tensor M). $$

\begin{lemma}\label{appendix-lemma-C}
$\bar{C}$ is free over $E[\tau_0]$.
\end{lemma}
\begin{proof}
The Milnor diagonal $\Delta:A\to A\tensor A$ induces a right
$E[\tau_0]$-coaction given by
\begin{align*}
\psi(\tau_n) & = \tau_n \tensor 1 + \xi_n \tensor \tau_0
\\\psi(\xi_n)  & = \xi_n\tensor 1
\end{align*}
and in particular, $\xi_n^i$ is primitive for all $i$ and
$\psi(\xi_n^i\tau_n)=\xi_n^i\tau_n\tensor 1 + \xi_n^{i+1}\tensor \tau_0$.
We have a $E[\tau_0]$-comodule decomposition of $C$:
\begin{align*}
C = \tensors_{n=1}^\iy \F_p[\xi_n,\tau_n]/\tau_n^2
& = \tensors_{n=1}^\iy \Big(\F_p[\xi_n,\tau_n]/\tau_n^2\Big)
\\& = \tensors_{n=1}^\iy \Big(\F_p\{ 1 \}\dsum \F_p\{\xi_n^i,\ \tau_n\xi_n^i \st
i\geq 0\}\Big)
\\ & = \tensors_{n=1}^\iy\Big(\F_p\{ 1 \}\dsum \dsums_{i=1}^\iy \F_p\{ \xi_n^i,
\xi_n^{i-1}\tau_n \}\Big)
\end{align*}
where all of the summands $\F_p\{ \xi_n^i,\xi_n^{i-1}\tau_n \}$ are free. Thus, $C =
\tensors_{n=1}^\iy (\F_p\{ 1 \}\dsum F_n) = \F_p\{ 1 \}\dsum F$ for free
comodules $F$ and $F_n$, and hence $\bar{C}$ is free.
\end{proof}
If $F$ is free over $E[\tau_0]$, then $\Ext_{E[\tau_0]}^*(\F_p,F)$ is concentrated
in homological degree zero. Thus $E_1^{s,t} =0$ unless $s=0$ or $t=0$. For
degree reasons, the $a_0$-localized Cartan-Eilenberg spectral sequence
$$ a_0^{-1}E_1^{s,t} = a_0^{-1} \Ext_{E[\tau_0]}^*(\F_p,\bar{C}^{\tensor
s}\tensor M) \implies a_0^{-1} \Ext_A(\F_p,M)$$
converges. Furthermore, $a_0^{-1} E_1^{s,t} = 0$ for $s>0$, and the
$a_0$-localized spectral sequence collapses at $E_1$ for degree reasons, giving
the isomorphism
\begin{align*}
a_0^{-1}E_1^{0,*}  & = a_0^{-1}\Ext^*_{E[\tau_0]}(\F_p,M) \isom a_0^{-1}\Ext_A(\F_p,M). \qedhere
\end{align*}
\end{proof}

\section{The Cartan-Eilenberg spectral sequence}\label{section:CESS}
\subsection{Notation and preliminaries}
Throughout the paper, $\Gamma$ will be a Hopf algebra over a commutative ring
$k$, and $M$ and $N$ will be $\Gamma$-comodules. Coproducts will be denoted by
$\Delta$, and comodule coactions will be denoted by $\psi$.
\begin{random}{Notation}
We write $\sum m'\tensor m''\colonequals\psi(m)$  and $\sum \gamma'\tensor
\gamma''\colonequals\Delta(\gamma)$ for $m\in M$ and $\gamma\in \Gamma$ when
there is no ambiguity which coaction is in play.
\end{random}

An essential technical point in this paper is the comparison between two
different $\Gamma$-comodule structures on a tensor product of
$\Gamma$-comodules; we use the following nonstandard notation to clarify which
structure is in play at a given time.

\begin{definition}\label{left-diagonal-coaction}
Let $M$ and $N$ be left $\Gamma$-comodules, with coaction denoted by $\psi(m)=\sum
m'\tensor m''$ and $\psi(n)=\sum n'\tensor n''$. There are two natural ways to put a
$\Gamma$-comodule structure on their tensor product $M\tensor N$: the
\emph{left coaction} $M\tensor N \to \Gamma\tensor (M\tensor N)$ is given by
$m\tensor n\mapsto \sum m'\tensor m''\tensor n$, and the \emph{diagonal
coaction} is given by $m\tensor n \mapsto \sum m'n'\tensor m''\tensor n''$. To
distinguish these, we write $M\tensorL N$ for the tensor product $M\tensor N$
endowed with the left $\Gamma$-coaction, and $M\tensorD N$ for the diagonal
coaction.

For a pair of right $\Gamma$-comodules one can analogously define the right and
diagonal coactions, denoted $\tensorR$ and $\tensorD$, respectively.
\end{definition}
These constructions are isomorphic in the following special case:
\begin{lemma}\label{shear}
If $M$ is a left $\Gamma$-comodule, there is an isomorphism $S:\Gamma\tensorD
M\to \Gamma\tensorL M$ (called the \emph{shear isomorphism}) given by:
\begin{align*}
S: a\tensor m & \mapsto \sum am' \tensor m''
\\S^{-1}: a\tensor m & \mapsto \sum ac(m')\tensor m''
\end{align*}
where $c$ is the antipode on $\Gamma$. Analogously, if $M$ is a right
$\Gamma$-comodule, there is an isomorphism $S_c:M\tensorD \Gamma\to M\tensorR
\Gamma$ given by:
\begin{align*}
S_c: m\tensor a & \mapsto \sum m'\tensor m''a
\\S_c^{-1}: m\tensor a & \mapsto \sum m'\tensor c(m'')a.
\end{align*}
\end{lemma}


The proof is straightforward.
Now suppose $\Phi = \Gamma\cotensor_\Sigma k$ for a Hopf algebra $\Sigma$ such
that $\Gamma$ is injective as a $\Sigma$-comodule.
\begin{lemma}\label{shear-cotensor}
Let $M$ be a $\Gamma$-comodule. Then $\Gamma\cotensor_\Sigma M\subset
\Gamma\tensorL M$ inherits a left $\Gamma$-comodule structure, and
the shear isomorphism $S:\Gamma\tensorD M\to \Gamma\tensorL M$ restricts to an
isomorphism $$\Phi\tensorD M\too{\isom} \Gamma\cotensor_\Sigma
M.$$
The shear isomorphism $S_c:M\tensorD \Gamma\to M\tensorR \Gamma$ restricts to an
isomorphism $M\tensorD \Phi\too{\isom} M\cotensor_\Sigma\Gamma$.
\end{lemma}

Using this, we produce a useful variant of the usual change of rings theorem
$$\Ext_\Gamma(M,\Gamma\cotensor_\Sigma
N)\isom \Ext_\Sigma^*(M,N)$$ (see, e.g., \cite[\S VI.4]{cartan-eilenberg}).

\begin{corollary}[Change of rings theorem] \label{change-of-rings}
Let $M$ be a right $\Gamma$-comodule and $N$ a left $\Gamma$-comodule, and let $\Phi = \Gamma\cotensor_\Sigma N$. Then there is an isomorphism
$$ \Ext_\Gamma^*(M, \Phi\tensorD N) \isom \Ext_\Sigma^*(M,N). $$
\end{corollary}
Both change of rings statements hold for Cotor in addition to comodule $\Ext$.

\subsection{Background: Classical Cartan-Eilenberg spectral
sequence}\label{section:old-CESS} We begin by reviewing the classical
construction of the Cartan-Eilenberg spectral sequence, expressed in the
language of coalgebras. Following the treatment in \cite[A1.3.14]{green}, we
will describe a spectral sequence that converges to $\Cotor_\Gamma(M,N)$. Many
other treatments define the Cartan-Eilenberg spectral sequence in the coalgebra
context using comodule Ext, but this does not matter for our cases of interest
due to the isomorphism
$$ \Hom_\Gamma(k,M) \isom k\cotensor_\Gamma M $$
which implies $\Ext_\Gamma(k,M) \isom \Cotor_\Gamma(k,M)$.
We choose Cotor because it is easier to work with: both slots are covariant, and
it can be computed using an injective resolution of either side (just as
projective resolutions are more convenient for modules, injective resolutions
are more convenient for comodules).

Given an extension of Hopf algebras $$\Phi\to
\Gamma\to \Sigma$$ over a commutative ring $k$ (so in particular $\Phi = \Gamma\cotensor_\Sigma k$), a right
$\Gamma$-comodule $M$, and a left $\Phi$-comodule $N$, the Cartan-Eilenberg
spectral sequence for computing $\Cotor_\Gamma(M,N)$ arises from the double
complex $(\Gamma\text{-resolution of }M)\cotensor_\Gamma (\Phi\text{-resolution of
}N)$. Our choice of injective resolution is the cobar resolution, which we will
describe in some detail because an essential technical point in our
spectral sequence construction comes down to the difference between two versions
of the cobar resolution ($\DD*(N)$ and $\DL*(N)$) which are isomorphic via the
shear isomorphism.

\begin{definition}\label{def-DL}
Define the cobar resolution $\DL*(N)$ of $N$ to be the chain complex associated
to the augmented cosimplicial object
\begin{align}
\label{cobar-cosimplicial-L}\xymatrix{
N\ar[d]^-\hteq
\\\hspace{-25pt}\DL\bullet(N)=\big(\Gamma\tensorL N\ar@<-1.2ex>[r]_-{\psi}\ar@<1.2ex>[r]^-\Delta & \Gamma\tensorL\Gamma\tensorL
N\ar[l]|-\epsilon\ar@<2.4ex>[r]^-{\Delta_1}\ar@<-2.4ex>[r]_-{\psi}\ar[r] &
\Gamma\tensorL \Gamma\tensorL\Gamma\tensorL
N\ar@<-1.2ex>[l]|-{\epsilon_2}\ar@<1.2ex>[l]|-{\epsilon_1} & \dots\hspace{25pt}\big).
}
\end{align}
Here the codegeneracies $\epsilon_i$ come from applying the coaugmentation
$\epsilon$ to the $i^{th}$ spot, and the coface maps $\Delta_i:\Gamma^{\tensor
n}\tensor N\to \Gamma^{\tensor n+1}\tensor N$ for $1\leq i\leq n$ come from
applying $\Delta$ to the $i^{th}$ slot; the last coface map comes from
the coaction $\psi:N\to \Gamma\tensor N$.
For a right $\Gamma$-comodule $M$, let
$\DR*(M)$ denote the analogous resolution $M\tensor \Gamma^{\tensor *}\tensorR
\Gamma$, and similarly for $\CLR*(M)$.

Let the \emph{(non-normalized) cobar complex $\CL*(M,N)$} be the complex
$M\cotensor_\Gamma \DL*(N)$.

There are also normalized versions $\mathcal{N}\DL*(N)$ (with terms $\Gamma\tensorL
\bar{\Gamma}^{\tensor *}\tensor N$) and
$\mathcal{N}\CL*(M,N) = M\tensor \bar{\Gamma}^{\tensor *}\tensor N$ (with terms
$M\cotensor_\Gamma(\Gamma\tensorL \bar{\Gamma}^{\tensor *}\tensor N)$), obtained by
applying the normalization functor $\mathcal{N}: \Ch\to \Ch$ defined on terms by
$$\mathcal{N}A^n = \intss_{i=0}^{n-1}\ker(d_i:A^{n+1}\to A^n)\subset A^n.$$
(Here $\bar{\Gamma}$ denotes $\ker(\epsilon:\Gamma\to k)$ but we will later also use
that symbol to denote the quotient $\coker(\eta_L:k\to \Gamma)$.)
\end{definition}

For our purposes, the classical Cartan-Eilenberg spectral sequence is the
spectral sequence associated to the double complex
$\mathcal{N}\DR*(M)\cotensor_\Gamma \mathcal{N}\DLPhi*(N)$.
\begin{equation}\label{classical-cess} \scalebox{0.80}{\xymatrix@C=50pt{
 & \vdots\ar[d] & \vdots\ar[d]
\\ \dots\ar[r]& (M\tensor \bar{\Gamma}^{\tensor
t}\tensor \Gamma)\cotensor_\Gamma(\Phi\tensor \bar{\Phi}^{\tensor s}\tensor
N)\ar[r]^-{(-1)^t\Id\tensor d_\Phi}\ar[d]_-{d_\Gamma\tensor \Id} & (M\tensor \bar{\Gamma}^{\tensor t}\tensor
\Gamma)\cotensor_\Gamma(\Phi\tensor \bar{\Phi}^{\tensor
s+1}\tensor N)\ar[d]^-{d_\Gamma\tensor \Id}\ar[r] & \dots
\\ \dots\ar[r] & (M\tensor \bar{\Gamma}^{\tensor
t+1}\tensor \Gamma)\cotensor_\Gamma(\Phi\tensor \bar{\Phi}^{\tensor s}\tensor
N)\ar[r]^-{(-1)^{t+1}\Id\tensor d_\Phi}\ar[d]
& (M\tensor \bar{\Gamma}^{\tensor t+1}\tensor \Gamma)\cotensor_\Gamma(\Phi\tensor
\bar{\Phi}^{\tensor s+1}\tensor N)\ar[d]\ar[r] & \dots
\\ &  \vdots & \vdots
}}\end{equation}
Taking homology in the vertical direction first, one has
\begin{align*}
E_1^{s,t}  & = \Cotor^t_\Gamma(M,\,\Phi\tensor \bar{\Phi}^{\tensor s}\tensor N)
\\ & \isom \Cotor^t_\Gamma(M,\,(\Gamma\cotensor_\Sigma k)\tensor
\bar{\Phi}^{\tensor s}\tensor N)
\\ & \isom \Cotor^t_\Sigma(M,\, \bar{\Phi}^{\tensor s}\tensor N)
\end{align*}
where the last isomorphism is by the change of rings theorem. For the spectral
sequence that starts by taking homology in the horizontal direction first,
exactness of the functor $(M\tensor \bar{\Gamma}^{\tensor t}\tensor
\Gamma)\cotensor_\Gamma -$
gives $$E_1^{*,t} \isom H^*((M\tensor \bar{\Gamma}^{\tensor t}\tensor
\Gamma)\cotensor_\Gamma
(\Phi\tensor \bar{\Phi}^{\tensor *}\tensor N)) \isom (M\tensor
\bar{\Gamma}^{\tensor t}\tensor \Gamma)\cotensor_\Gamma H^*(\Phi\tensor
\bar{\Phi}^{\tensor *}\tensor N)$$
and by the exactness of the resolution $\Phi\tensor \bar{\Phi}^{\tensor
*}\tensor N$ of $N$,
this is concentrated in degree zero as $(M\tensor \bar{\Gamma}^{\tensor
t}\tensor \Gamma)\cotensor_\Gamma N$. The $E_2$ page then takes cohomology in
the $t$ direction, obtaining $E_2 \isom E_\iy \isom \Cotor_\Gamma(M,N)$.
The Cartan-Eilenberg spectral sequence is the vertical-first spectral sequence,
and we have just shown that it has
$$ E_1^{s,t} = \Cotor^t_\Sigma(M,\bar{\Phi}^{\tensor s}\tensor N) \implies
\Cotor^{s+t}_\Gamma(M,N). $$
If $\Phi$ has trivial $\Sigma$-coaction, then we have $E_1^{s,t} \isom
\Cotor^t_\Sigma(M,N)\tensor \bar{\Phi}^{\tensor s}$, whose cohomology is:
$$ E_2 = \Cotor^s_\Phi(k,\Cotor^t_\Sigma(M,N)). $$
The spectral sequence converges because it is a first-quadrant double complex
spectral sequence.
\begin{remark}
The $E_2$ page is independent of the $\Phi$-resolution of $N$ and the
$\Gamma$-resolution of $M$, but the $E_1$ page does depend on the
$\Phi$-resolution of $N$.
\end{remark}

\subsection{Weakening the hypotheses}\label{subsection:new-CESS} We define a
related construction that makes sense when $\Phi$ is not a coalgebra. More
precisely, let $\Gamma$ be a Hopf algebra and let $\Phi$ be a
$\Gamma$-comodule-algebra. The main issue with defining an analogue of
\eqref{classical-cess} is that the cosimplicial object $\DLPhi\bullet(N)$ is not
defined, because the coface maps would be defined in terms of the coproduct on
$\Phi$.

To remedy this, we turn to a different construction of the cobar complex which
is defined for unital algebras, and which is isomorphic to Definition
\ref{def-DL} when we are working with a Hopf algebra. For a $\Gamma$-comodule
$N$, define the resolution $\DD*(N)$ of $N$ to be the chain complex associated
to the augmented cosimplicial object
\begin{align}
\label{cobar-cosimplicial-Delta} \xymatrix{
N\ar[d]^-\hteq
\\\hspace{-25pt}\DD\bullet(N)=\big(\Gamma\tensorD
N\ar@<-1.2ex>[r]_-{\eta_2}\ar@<1.2ex>[r]^-{\eta_1} & \Gamma\tensorD\Gamma\tensorD
N\ar[l]|-\mu\ar@<2.4ex>[r]^-{\eta_1}\ar@<-2.4ex>[r]_-{\eta_3}\ar[r] & \Gamma\tensorD
\Gamma\tensorD\Gamma\tensorD N\ar@<-1.2ex>[l]|-{\mu_1}\ar@<1.2ex>[l]|-{\mu_2}
 & \dots\hspace{25pt}\big)
}\end{align}
where the codegeneracies $\mu_i$ are multiplication of the
$i^{th}$ and $(i+1)^{st}$ copies of $\Gamma$, and the coface maps $\eta_i$ are
given by insertion of 1 into the $i^{th}$ spot.
Substituting $\Phi$ for $\Gamma$ above, if $N$ is a $\Gamma$-comodule we may
still define $\DDPhi*(N)$, a complex of $\Gamma$-comodules free over $\Phi$ and
quasi-isomorphic to $N$.

We can also describe $\DDPhi\bullet(N)$ in a more natural way. Since $\Phi$ is a
monoid object in $\Comod_\Gamma$, we can define the category $\Mod_\Phi$ of
$\Phi$-modules in $\Comod_\Gamma$. There is a free-forgetful adjunction $$
F_\Phi: \Comod_\Gamma \longleftrightarrows \Mod_\Phi : U $$ where $F_\Phi(N)=
\Phi\tensorD N$. Then $\DDPhi\bullet(N)$ is the cosimplicial object associated
to the monad $UF_\Phi$.

\begin{definition}\label{cobar-complex}
The \emph{(non-normalized) cobar complex} $\CD*(M,N)$ is the complex
$M\cotensor_\Gamma \DD*(N)$. Similarly, define
$\CL*(M,N)=M\cotensor_\Gamma \DL*(N)$.
\end{definition}

\begin{proposition}\label{shear-iso-on-cosimplicial-objects}
The shear isomorphism (Lemma \ref{shear}) gives rise to an isomorphism of
cosimplicial objects $\DD\bullet(N)\to \DL\bullet(N)$, and hence isomorphisms of
chain complexes $\DD*(N)\to \DL*(N)$ and $\CD*(N)\to \CL*(N)$.
\end{proposition}
In the appendix, we write out explicit formulas for these isomorphisms.


\begin{definition}\label{new-CESS}
If $\Gamma$ is a Hopf algebra, $\Phi$ is a $\Gamma$-comodule-algebra, and $M$
and $N$ are $\Gamma$-comodules, define the Cartan-Eilenberg spectral sequence to
be the spectral sequence associated to the double complex
$$ (\mathcal{N}\DD{*}(M))\,\cotensor_\Gamma\,(\mathcal{N}\DDPhi*(N)). $$
\end{definition}
The spectral sequence is unchanged starting at $E_1$ if we replace the
complex on the right by a chain-homotopic one, 
and in Section \ref{section:CESS-vs-filtration} we will
find it more convenient to use the complex
\begin{equation}\label{I-CESS}
\DD{*}(M)\,\cotensor_\Gamma\,(\mathcal{N}\DDPhi{*}(N)).
\end{equation}
By definition, we have the $E_1$ term
$$ E_1^{s,t} = \Cotor_\Gamma^t(M,\mathcal{N}\DDPhi{*}(N)) $$
and it converges to $\Cotor_\Gamma(M,N)$ as with the usual construction of the
Cartan-Eilenberg spectral sequence. As in the classical case, if $\Phi =
\Gamma\cotensor_\Sigma k$ for some coalgebra $\Sigma$, then by the version of
the change of rings theorem in Corollary \ref{change-of-rings} we may
write \begin{equation}\label{Cotor-E_1}E_1^{s,t} \isom \Cotor_\Gamma^t (M,\Phi\tensorD
(\bar{\Phi}^{\tensor s}\tensorD N))\isom \Cotor_\Sigma^t(M, \bar{\Phi}^{\tensorD
s}\tensorD N).\end{equation}
In particular, if $M=k$, then $E_1^{s,t}\isom \Ext^t_\Sigma(k,\bar{\Phi}^{\tensor
s}\tensor N)$ and the spectral sequence converges to $\Ext_\Gamma(k,N)$.
\begin{remark}
If $\Phi$ did have a coalgebra structure, we can also define the spectral
sequence in Section \ref{section:old-CESS}, and by Proposition
\ref{shear-iso-on-cosimplicial-objects}
the two spectral sequences are isomorphic via the shear isomorphism.
\end{remark}

\begin{remark}
Davis and Mahowald \cite{davis-mahowald-A(2)} studied a Cartan-Eilenberg type
spectral sequence in the setting $\Gamma = A(2)_*$, $\Sigma= A(1)_*$ (so $\Phi =
\Gamma\cotensor_\Sigma k$ is not a sub-Hopf algebra of $\Gamma$). Instead of the
cobar-inspired resolution described here, they work with a minimal resolution,
which is more convenient for computational purposes. Bruner and Rognes
\cite{bruner-rognes} construct a spectral sequence converging to
$\Ext_\Gamma(k,N)$ given the data of a $\Gamma$-comodule algebra resolution of
$k$ of the form $\Phi\tensorD R^*$ where $\Phi = \Gamma\cotensor_\Sigma k$ for a
Hopf algebra $\Sigma$ and $R^*$ is a sequence of $\Sigma$-comodules. They show
that their spectral sequence is multiplicative, given suitable multiplicative
properties of $R^*$. The construction in this section can be seen as a special
case of theirs, where $R^n = \bar{\Phi}^{\scriptscriptstyle\tensorD n}$. Our construction does not
use the presentation of $\Phi$ as $\Gamma\cotensor_\Sigma k$ except to obtain
the nicer form of the $E_1$ page in \eqref{Cotor-E_1}.
\end{remark}

\section{Margolis-Palmieri Adams spectral sequence}\label{section:MPASS}
\subsection{Background: Adams spectral sequence in $\Stable(\Gamma)$}
Given a finite spectrum $X$ and a ring spectrum $E$, the classical Adams
spectral sequence is the spectral sequence obtained by applying $\pi_*(-)$ to
the tower of fibrations
\begin{equation}\label{classical-adams-diagram} \xymatrix{
X\ar[d] & \bar{E}\sm X\ar@{.>}[l]\ar[d] & \bar{E}\sm \bar{E}\sm X\ar@{.>}[l] &
\dots\ar@{.>}[l]
\\E\sm X\ar[ru] & E\sm \bar{E}\sm X\ar[ru]
}\end{equation}
where $\bar{E}$ is the cofiber of
the unit map $S\to E$. If $E_*E$ is flat as an $E_*$-algebra, then the $E_2$
page is given by $\Ext_{E_*E}(E_*,E_*X)$.

This construction makes sense in the context of an arbitrary tensor triangulated
category $(\mathcal{C},\tensor,\Id)$. Given a ring object $E$ and another object
$X$ of $\mathcal{C}$, let $\bar{E}$ be the cofiber of the unit map $\Id\to E$.
Then one can construct the same tower of fibrations \eqref{classical-adams-diagram} and apply
the functor $\Hom_{\mathcal{C}}(\Id,-)$, giving rise to a spectral sequence
which, under favorable conditions, converges to (a completion of)
$\Hom_{\mathcal{C}}(\Id,X)$.

Following Palmieri \cite{palmieri-book}, we study this generalized Adams spectral
sequence in the case $\mathcal{C}=\Stable(\Gamma)$, the category whose objects
are unbounded cochain complexes of injective $\Gamma$-comodules and whose
morphisms are chain complex morphisms modulo chain homotopy. The reason to work
in this setting is the fact that $$\Hom_{\Stable(\Gamma)}(M,N) =
\Ext_\Gamma(M,N)$$ for $\Gamma$-comodules $M$ and $N$ (we abuse notation by
identifying $M$ with its image under the functor $\Comod_\Gamma\to
\Stable(\Gamma)$ given by taking injective resolutions).
Thus one can use techniques from homotopy theory to study Ext groups.


\begin{remark}
The reader may wonder why we have chosen $\Stable(\Gamma)$ instead of the
more familiar derived category $D(\Gamma)$, as there is also an identification
$\Hom_{D(\Gamma)}(M,N) = \Ext_\Gamma(M,N)$. The reason is that
$\Stable(\Gamma)$ is a better setting for studying localized Ext groups: if
$x\in \Ext_\Gamma(k,k)$ is a non-nilpotent element, $M$ and $N$ are
$\Gamma$-comodules,
and $x^{-1}N$ is the colimit of multiplication by $x$ in $\Stable(\Gamma)$, then
$$ \Hom_{\Stable(\Gamma)} (M, x^{-1}N) = x^{-1}\Ext_\Gamma(M,N)$$
but the analogous statement in $D(\Gamma)$ is not guaranteed to hold.
As we care most about the constructions defined in this paper after localizing
at a non-nilpotent element, and so this property is essential to ensuring that
this localization is well-behaved.
\end{remark}


The Adams spectral sequence in this setting was first studied by Margolis
\cite{margolis} in the case where $\Gamma$ is the dual Steenrod algebra, work
which was extended and generalized by Palmieri. If $E\in \Stable(\Gamma)$ is a
ring object (for example, a $\Gamma$-comodule algebra) and $X\in
\Stable(\Gamma)$ we refer to the $E$-based Adams spectral sequence in
$\Stable(\Gamma)$ computing $\Hom_{\Stable(\Gamma)}(k,X)$ as the \emph{$E$-based
Margolis-Palmieri Adams spectral sequence} (MPASS). Analogously to the classical
Adams flatness condition, if $\Hom_{\Stable(\Gamma)}(k,E\tensorD E)$ is flat over
$\Hom_{\Stable(\Gamma)}(k,E)$, then the $E_2$
page of the MPASS is
\begin{equation}\label{MPASS-E_2} E_2 \isom \Ext_{\Hom_{\Stable(\Gamma)}(k,E\tensorD
E)}(\Hom_{\Stable(\Gamma)}(k,E),\ \Hom_{\Stable(\Gamma)}(k,E\tensorD X)).
\end{equation}
Palmieri \cite[Proposition 1.4.3]{palmieri-book} identifies finiteness
conditions on $E$ and $X$ under which the MPASS converges to
$\Hom_{\Stable(\Gamma)}(k,X)$.
The motivating application is the case where $\Gamma$ is the dual Steenrod
algebra, $X = H_*(Y)$ for a finite spectrum $Y$, and $E$ is a subalgebra of $A$.
Then the MPASS
$$ E_2 \isom \Ext_{\Ext_A(k,E\tensor E)}(\Ext_A(k,E), \Ext_A(k,E\tensor
H_*Y))\implies \Ext_A(k,H_*Y) $$
converges to the $E_2$ page of the Adams spectral sequence for $\pi_*Y$.

\subsection{Comparison: MPASS vs. Cartan-Eilenberg spectral sequence}
\begin{theorem}\label{CESS=MPASS}
Given a left $\Gamma$-comodule-algebra $\Phi$ and a left $\Gamma$-comodule $N$, 
the Cartan-Eilenberg spectral sequence
$$ \ED{s,*}1 = H^*\big(\DD*(k)\cotensor_\Gamma (\mathcal{N}\DDPhi{s}(N))\big) \implies \Cotor^*_\Gamma(k,N)\isom \Ext_\Gamma^*(k,N)$$
coincides starting at $E_1$ with the $\Phi$-based MPASS
$$ E_1^{s,*} = \Ext_\Gamma^*(k,\Phi\tensorD \bar{\Phi}^{\tensorD s}\tensorD N)
\implies \Ext_\Gamma^*(k,N).$$
\end{theorem}
\begin{proof}
Given a chain complex $A^*$, let $\mathcal{Q}A$ denote the quotient chain
complex whose terms are given by $\mathcal{Q}A^n = A^n\big/ \sum_{i=1}^n
\im(d^i:A^{n-1}\to A^n)$. In particular, $\mathcal{Q}\DDPhi *(N) \isom
\Phi\tensorD \bar{\Phi}^{\tensorD *}\tensorD N$.
It is a general fact (see \cite[Theorem III.2.1 and Theorem
III.2.4]{goerss-jardine} for the dual version) that there is an isomorphism of
chain complexes $\mathcal{N}^*A \too{\isom} \mathcal{Q}^*A$, so instead of the double
complex $\DD*(k)\cotensor_\Gamma (\mathcal{N}\DDPhi{s}(N))$ we may use $$
\DD{*}(k)\cotensor_\Gamma (\mathcal{Q}\DDPhi*(N)) =
\Gamma^{\tensorD t+1}\cotensor_\Gamma (\Phi\tensorD \bar{\Phi}^{\tensorD
s}\tensorD N).$$

We will express the exact couples for both spectral sequences as coming from
applying $\Ext_\Gamma(k,-)$ to fiber sequences in $\Stable(\Gamma)$, and show
that there is a quasi-isomorphism connecting those fiber sequences.
We begin by describing the exact couple for the
Cartan-Eilenberg spectral sequence more explicitly.
Let $T^*$ be the total complex defined by $T^n = \dsums_{s+t=n}
\Gamma^{\tensorD t+1}\cotensor_\Gamma (\Phi\tensorD \bar{\Phi}^{\tensorD
s}\tensorD N)$. The Cartan-Eilenberg spectral sequence arises from the
filtration $F^s$ on this total complex defined by:
$$ F^{s_0} T^n = \dsums_{s+t=n\atop s\geq s_0}
\Gamma^{\tensorD t+1}\cotensor_\Gamma (\Phi\tensorD \bar{\Phi}^{\tensorD
s}\tensorD N).$$
For the associated graded we have:
\begin{align*}
F^{s_0}/F^{s_0+1}T^n  & = \Gamma^{\tensorD
n-s_0+1}\cotensor_\Gamma (\Phi\tensorD \bar{\Phi}^{\tensorD s_0}\tensorD N)
\\ H^*(F^{s_0}/F^{s_0+1}T^*) & = \Cotor^*_\Gamma(k,\Phi\tensorD
\bar{\Phi}^{\tensorD s}\tensorD N).
\end{align*}
By definition, the Cartan-Eilenberg spectral sequence arises from the exact
couple
\begin{equation}\label{CESS-exact-couple} \xymatrix{
H^*(F^s T^*)\ar[rd] &  & H^*(F^{s+1}T^*)\ar[ll]
\\ & H^*(F^s/F^{s+1}T^*).\ar@{.>}[ru]
}\end{equation}


On the other hand, the MPASS comes from the exact couple obtained by applying
the functor $\Ext_\Gamma(k,-)$ to the cofiber sequence
\begin{equation}\label{MPASS-cofiber-sequence} \bar{\Phi}^{\tensorD s+1}\tensorD
N\to (\bar{\Phi}^{\tensorD s}\tensorD N)[1]\to (\Phi\tensorD\bar{\Phi}^{\tensor
s}\tensorD N)[1]. \end{equation} in $\Stable(\Gamma)$. Since $\Ext_\Gamma(k,-)\isom
\Cotor_\Gamma(k,-)$, this is isomorphic to the exact couple
\begin{equation}\label{MPASS-exact-couple} \xymatrix@C=-20pt{
H^*(\DD*(k)\cotensor_\Gamma(\bar{\Phi}^{\tensorD s}\tensorD
N))\ar[rd] && \ar[ll] H^*(\DD*(k)\cotensor_\Gamma(\bar{\Phi}^{\tensorD
s+1}\tensorD N))
\\ & H^*(\DD*(k)\cotensor_\Gamma(\Phi\tensorD
\bar{\Phi}^{\tensorD s}\tensorD N))\ar@{.>}[ru]
}\end{equation}

We claim that \eqref{MPASS-exact-couple} and \eqref{CESS-exact-couple} are
isomorphic exact couples. For the same reason that $H^*(T^*)\isom
\Ext_\Gamma(k,M)$, we have $H^*(F^sT^*) \isom \Ext_\Gamma(k,\bar{\Phi}^{\tensor
s}\tensor M)$. Moreover, there is a map $\DD*(k)\cotensor_\Gamma
(\bar{\Phi}^{\tensorD s}\tensorD M)\to F^sT^*$ induced by the unit map
$\bar{\Phi}^{\tensorD s}\tensorD N\to \Phi\tensorD\bar{\Phi}^{\tensorD
s}\tensorD N$ that induces this isomorphism in cohomology compatibly with the
rest of the exact couple.
\end{proof}

The comparison statement shows that the $E_2$ page of the Cartan-Eilenberg
spectral sequence coincides with the  MPASS $E_2$ page \eqref{MPASS-E_2}.
\begin{corollary}\label{MPASS-flatness}
If $\Ext_\Gamma(k,\Phi\tensor \Phi)$ is flat as a module over
$\Ext_\Gamma(k,\Phi)$, then the Cartan-Eilenberg spectral sequence of Definition
\ref{new-CESS} has $E_2$ term given by
$$ E_2^{**}\isom \Ext_{\Ext_\Gamma(k,\Phi\tensor \Phi)}^*(\Ext_\Gamma(k,\Phi),
\Ext_\Gamma^*(k,\Phi\tensorD N)) .$$
If $\Phi= \Gamma\cotensor_\Sigma k$ for some coalgebra $\Sigma$, then by the
change of rings theorem (Corollary \ref{change-of-rings}) this has the form
$$ E_2^{**} = \Ext_{\Ext_\Sigma(k,\Phi)}(\Ext_\Sigma(k,k),\
\Ext_\Sigma(k,N)). $$
For $x\in \Ext_\Gamma(k,k)$, the $x$-localized Cartan-Eilenberg spectral
sequence has $E_2$ term
$$ \Ext_{x^{-1}\Ext_\Sigma(k,\Phi)} (x^{-1}\Ext_\Sigma(k,k),\ x^{-1}\Ext_\Sigma(k,N)).$$
\end{corollary}
Note that, for the localized spectral sequence, one must additionally check
convergence.

\section{Cartan-Eilenberg vs. filtration spectral sequence}\label{section:CESS-vs-filtration}
It is a classical fact \cite[\S2.3]{adams-hopf-invariant} that the
Cartan-Eilenberg spectral sequence associated to the Hopf extension $\Phi\to
\Gamma\to \Sigma$ computing $\Cotor_\Gamma(M,N)$ coincides with
a filtration spectral sequence on the cobar complex $C_\Gamma(M,N)$ defined by
$$ F^s C_\Gamma^n(M,N) = \{ m[a_1|\dots|a_n]\nu \in C_\Gamma^n(M,N)\st \#(\{
a_1,\dots,a_n \}\ints G) \geq s \} $$
where
$$ G \colonequals \ker(\Gamma\to \Sigma). $$
As $G$ is an ideal in $\Gamma$ and the cobar complex $C_\Gamma^*(k,k)$ is a ring
under the concatenation product, one can say this filtration of $C_\Gamma^*(M,N)
= M\tensor C_\Gamma^*(k,k)\tensor N$ comes from the $G$-adic filtration of
$C_\Gamma^*(k,k)$. In this section, we adopt the notation of the previous
sections, but also impose the additional condition that $\Phi =
\Gamma\cotensor_\Sigma k$ where $\Gamma\to\Sigma$ is a map of Hopf algebras.

Let $E_r^{**}$ denote this filtration spectral sequence, and let $\ED{**}r$
denote the generalized Cartan-Eilenberg spectral sequence described in
Definition
\ref{new-CESS}. Adapting an argument for the classical Cartan-Eilenberg spectral
sequence, we will
show that these agree starting at $r=1$. As a double complex spectral sequence
can be viewed as a filtration spectral sequence on the total complex, it
suffices to show the following:

\begin{theorem}
There is a filtration-preserving chain map $$\theta:
\dsums_{s+t=n}(M\tensorD \Gamma^{\tensorD t+1})\cotensor_\Gamma (\mathcal{N}\DDPhi s(N))\tto C_\Gamma^n(M,N)$$
whose induced map of spectral sequences $\ED{**}r\to E_r^{**}$ is an
isomorphism on $E_1$.
\end{theorem}
We begin by defining the comparison map $\theta$.

\begin{definition}
Let $\til{\theta}$ denote the composition
\begin{align*}
\til{\theta}: (M\tensorD \Gamma^{\tensorD t+1})\cotensor_\Gamma
(\Phi^{\tensorD s+1}\tensorD N)& \ttoo{S_c^n\tensor S^n} (M\tensor \Gamma^{\tensor t}\tensorR
\Gamma)\cotensor_\Gamma (\Gamma\tensorL \Gamma^{\tensor s}\tensor N)
\\ & \ttoo{e}M\tensor \Gamma^{\tensor s+t}\tensor N
\end{align*}
where $S_c^n$ is $n$-fold composition of the shear isomorphism $S_c:M\tensorD
\Gamma\to M\tensorR \Gamma$, $S^n$ is the $n$-fold composition of the iterated
shear isomorphism $S:\Gamma\tensorD N\to \Gamma\tensorL N$,
and $e$ is given by $$(m|a_1|\dots|a_t|a)\tensor
(b|b_1|\dots|b_s|n)\mapsto \epsilon(ab)m|a_1|\dots|a_t|b_1|\dots|b_s|n.$$
Define $\theta$ to be the restriction of $\til{\theta}$ to
$(M\tensorD \Gamma^{\tensorD t+1})\cotensor_\Gamma (\mathcal{N}\DDPhi s(N))$.
\end{definition}
In Lemma \ref{shear-restriction-G}, we will show that this restriction lands in
$(M\tensorD \Gamma^{\tensorD t+1})\cotensor_\Gamma (\Gamma\cotensor_\Sigma
G(s)\cotensor_\Sigma N)$, where $$G(s)\colonequals
\u{G\cotensor_\Sigma \dots \cotensor_\Sigma G}_s.$$


We will see that $E_0^{0,*}(M,N)$ is easy to describe (and in particular it is
easy to show that $\theta$ induces an isomorphism $\ED{0,*}0(M,N)\isom
E_0^{0,*}(M,N)$), and most of the work involves identifying $E_0^{s,*}(M,N)$
(for $s>0$) with $E_0^{0,*}(M,N')$ for a different comodule $N'$, in a way that
is compatible with a similar identification for $\ED{s,*}0$. More precisely, we
will show that there is a map $\beta$ of chain complexes making the
following diagram commute.
\begin{equation}\label{CESS-filtration-goal} \xymatrix@C=30pt{
(M\tensorD \Gamma^{\tensorD *})\cotensor_\Gamma\, \mathcal{N}\DDPhi
0(G(s)\cotensor_\Sigma N)\ar@{=}[r]\ar[d]_{\Id\tensor S^{-1}}^\isom &
\ED{0,*}0(M,G(s)\cotensor_\Sigma N)\ar[r]_-\hteq^-\theta &
E^{0,*}_0(M,G(s)\cotensor_\Sigma N)\ar[d]_-\hteq^-\beta
\\(M\tensorD \Gamma^{\tensorD *})\cotensor_\Gamma\, \mathcal{N}\DDPhi
s(N)\ar@{=}[r] & \ED{s,*}0(M,N)\ar[r]^-\theta & E^{s,*}_0(M,N)
}\end{equation}
It suffices to show the following:
\begin{enumerate} [label=(\arabic*)]
\item $\theta$ is a filtration-preserving chain map;
\item $S^{-1}$ gives rise to an isomorphism $\mathcal{N}\DDPhi
0(G(s)\cotensor_\Sigma N)\to \mathcal{N}\DDPhi s(N)$;
\item there exists a chain equivalence $\beta$ making the diagram commute;
\item $\theta$ is a chain equivalence for $s=0$.
\end{enumerate}
(1) says we have written down a filtration-preserving map between total
complexes, and (2)--(4) allow us to use the diagram to show that $\theta$ is a
chain equivalence for all $s\geq 0$.
We prove (1) in Lemma \ref{til-theta-chain-map} and Corollary
\ref{theta-filtration-preserving}, (2) in Corollary
\ref{goal-left-vertical-map}, (3) in Corollary/ Definition
\ref{corollary-definition}, and (4) in Proposition \ref{theta-s=0}.

Both the structure of the proof and the argument for (2) are
adapted from an argument attributed to Ossa appearing as
\cite[A1.3.16]{green}, showing that the classical Cartan-Eilenberg spectral
sequence coincides with the filtration spectral sequence under discussion. Our
proof is more complicated than Ossa's original, as the spectral sequence of
Definition \ref{new-CESS} generalizes the classical Cartan-Eilenberg spectral
sequence only after the iterated shear isomorphism has been applied. It is not
natural to describe the cobar filtration spectral sequence after applying the
isomorphism, so we must translate between the two contexts using explicit
formulas for the iterated shear isomorphism.

\begin{lemma}\label{til-theta-chain-map}
$\til{\theta}$ is a chain map $\dsums_{s+t=n} (M\tensorD \Gamma^{\tensorD
t+1})\cotensor_\Gamma \DDPhi s(N) \to C_\Gamma^n(M,N)$.
\end{lemma}
\begin{proof}
Since $S^n$ and $S_c^n$ are maps of chain
complexes of $\Gamma$-comodules, there is an induced map on the tensor
product of chain complexes
$$ (M\tensorD \Gamma^{\tensorD *+1})\tensor (\Phi^{\tensorD *+1}\tensorD N)\to
(M\tensor \Gamma^{*+1})\tensor (\Gamma^{\tensor *+1}\tensor N) $$
and since these are maps of chain complexes of $\Gamma$-comodules, this passes
to a map on the cotensor product
$$ (M\tensorD \Gamma^{\tensorD *+1})\cotensor_\Gamma (\Phi^{\tensorD
*+1}\tensorD N)\to (M\tensor \Gamma^{*+1})\cotensor_\Gamma (\Gamma^{\tensor
*+1}\tensor N) .$$
Then $\til{\theta}$ is formed by post-composing with the map
$$ e:(M\tensor \Gamma^{\tensor t+1})\cotensor_\Gamma (\Gamma^{\tensor s+1}\tensor
N) \to M\tensor \Gamma^{t+s}\tensor N  $$
which takes $m[a_1|\dots|a_t]a_{t+1}\tensor b_0[b_1|\dots|b_s]n \mapsto
\epsilon(a_{t+1}b_0)m[a_1|\dots|a_t|b_1|\dots|b_s]n$.
To see this is a chain map, it suffices to check the following diagram
commutes.
$$ \xymatrix@C=40pt{
(M\tensor \Gamma^{\tensor t+1})\cotensor_\Gamma (\Gamma^{\tensor s+1}\tensor
N)\ar[r]^-{\Id\tensor \epsilon\tensor
\Id}\ar[d]_-{d_{\attop{\text{double}\\\text{complex}}}} & M\tensor \Gamma^{\tensor t}\tensor 
\Gamma^{\tensor s}\tensor N\ar[d]^-{d_{\text{cobar}}}
\\ {\let\scriptstyle\textstyle \substack{(M\tensor \Gamma^{\tensor
t+1})\cotensor_\Gamma (\Gamma^{\tensor s+2}\tensor N) \\\dsum\ (M\tensor
\Gamma^{\tensor t+2})\cotensor_\Gamma (\Gamma^{\tensor s+1}\tensor N)}}\ar[r]^-{\Id\tensor \epsilon\tensor \Id} &
M\tensor \Gamma^{\tensor t+s+1}\tensor N 
}$$
This requires keeping track of signs: the double complex differential is
$d_\Gamma\tensor \Id + (-1)^t\Id\tensor d_\Phi$, or more explicitly:
\begin{align*}
a_0[a_1|\dots|a_t]a_{t+1}\tensor b_0[b_1|\dots|b_s]b_{s+1} & \mapsto \sum_i
(-1)^ia_0[\dots| a'_i|a''_i| \dots]a_{t+1}\tensor b_0[b_1|\dots|b_s]b_{s+1} 
\\ & \hspace{20pt}+ \sum_i (-1)^{i+t}a_0[a_1|\dots|a_t]a_{t+1}\tensor b_0[\dots|b'_i|b''_i|\dots]b_{s+1}
\itext{and the cobar differential is}
a_0[a_1|\dots|a_t|b_1|\dots|b_s]b_{s+1} & \mapsto \sum_i
(-1)^ia_0[a_1|\dots|a'_i|a''_i|\dots|b_1|\dots|b_s]b_{s+1} 
\\ &\hspace{20pt} + \sum_i
(-1)^{t+i}a_0[a_1|\dots|a_t|b_1|\dots|b'_i|b''_i|\dots|b_s]b_{s+1}.
\end{align*}
In particular, notice that, on the bottom left composition, the terms
corresponding to $a_0[\dots|a'_{t+1}]a''_{t+1}\tensor b_0[\dots]b_{s+1}$ cancel
in $M\tensor \Gamma^{\tensor t+s+1}\tensor N$ with the terms corresponding to
$a_0[\dots]a_{t+1}\tensor b'_0[b''_0|\dots]b_{s+1}$.
\end{proof}

While $\til{\theta}$ is not filtration-preserving, we will show that its
restriction to $(M\tensorD \Gamma^{\tensorD t+1})\cotensor_\Gamma\,\mathcal{N}\DDPhi s$ is.

\begin{lemma}\label{shear-restriction-G}
The iterated shear map $S:\Gamma^{\tensorD s+1}\tensorD N\to 
\Gamma^{\tensorL s+1}\tensor N$ restricts to an isomorphism $\mathcal{N}\DDPhi
s(N)\to \Gamma\cotensor_\Sigma G(s)\cotensor_\Sigma N$.
\end{lemma}
The proof is postponed to the appendix.

\begin{corollary}\label{theta-filtration-preserving}
$\theta$ is filtration-preserving.
\end{corollary}
\begin{proof}
This is a direct consequence of Lemma \ref{shear-restriction-G}.
\end{proof}
\begin{corollary}\label{goal-left-vertical-map}
There are isomorphisms
$$\mathcal{N}\DDPhi 0(G(s)\cotensor_\Sigma N)=\Phi\tensorD (G(s)\cotensor_\Sigma
N)\ttoo{S\tensor \Id} \Gamma\cotensor_\Sigma G(s)\cotensor_\Sigma N \ttoo{S^{-1}} \mathcal{N}\DDPhi s(N).$$
This gives the left vertical isomorphism in \eqref{CESS-filtration-goal}.
\end{corollary}

Our next task is to define the map $\beta$ in \eqref{CESS-filtration-goal} and
show it is a chain equivalence. Most of the work for that is done in Lemma
\ref{G-cobar-quotient}; the next lemma is helpful for that, and the result is
summarized in Corollary/ Definition \ref{corollary-definition}.
\begin{lemma}\label{M-cotensor-Sigma}
For fixed $s$, there is an isomorphism of complexes $F^s/F^{s+1}C_\Gamma(M,N)=E_0^{s,*}(M,N) \isom M\cotensor_\Sigma
E_0^{s,*}(M,\Sigma)\cotensor_\Sigma N$.
\end{lemma}
In particular, $E_0^{s,*}(M,N)$ only depends on the $\Sigma$-coaction on
$N$, not the full $\Gamma$-coaction. We will abuse notation by writing
$E_0^{s,*}(M,N)$ where $N$ has a $\Sigma$-coaction and not a $\Gamma$-coaction
(specifically, we do this for $N=G$).
\begin{proof}
We begin by showing that 
$F^s/F^{s+1}C_\Gamma(M,N)$ only depends on the $\Sigma$-coaction on
$N$: given $x= m[\gamma_1|\dots|\gamma_n]\nu$ in $F^sC_\Gamma(M,N)$, the term
$m[\gamma_1|\dots|\gamma_n|\nu']\nu''$ in $d(x)$ is in $F^{s+1}$ if $\nu'\in G$.
So, if we write $\psi(\nu) = \sum \underline{\nu}'|\underline{\nu}''$ for the coaction
$\psi:N\to \Sigma\tensor N$, we can say that $d(x)\equiv \sum
m[\gamma_1|\dots|\gamma_n|\underline{\nu}']\underline{\nu}''$ in $F^s/F^{s+1}C_\Gamma^{n+1}(M,N)$.

We have an isomorphism $\psi:N\too{\isom} \Sigma\cotensor_\Sigma N$ of
$\Sigma$-comodules, where the coaction on the right hand side is $\sigma\tensor
\nu \mapsto \sigma'\tensor \sigma''\tensor \nu$. This shows that the following diagram commutes
$$ \xymatrix{
E_0^{s,t}(M,N)\ar[r]^-{\psi}\ar[d]_-d & E_0^{s,t}(M,\Sigma)\cotensor_\Sigma
N\ar[d]^-d
\\E_0^{s,t+1}(M,N)\ar[r]^-\psi & E_0^{s,t+1}(M,\Sigma)\cotensor_\Sigma N
}$$
and so there is chain complex isomorphism $E_0^{s,*}(M,N) \isom
E_0^{s,*}(M,\Sigma)\cotensor_\Sigma N$ for every $s$.
\end{proof}

\begin{lemma}
[{\cite[A1.3.16]{green}}]\label{G-cobar-quotient}
The map
\begin{align*}
\delta:E_0^{s-1,*}(M,G) & \tto E_0^{s,*}(M,\Sigma)
\\m[a_1|\dots|a_{s-1}]g & \mapstto m[a_1|\dots|a_{s-1}|g']g''.
\end{align*}
is a chain equivalence, where $\sum g'\tensor g''$ is the image of $g\in G$
along the map $\Gamma\too{\Delta} \Gamma\tensor \Gamma\to \Gamma\tensor \Sigma$.
\end{lemma}

\begin{proof}
We introduce a second filtration $\til{F}^s$ which is defined only on
$C_\Gamma(M,\Gamma)$:
$$ \til{F}^sC^n_\Gamma(M,\Gamma)= \cu{m[\gamma_1|\dots|\gamma_n]\gamma\st
\text{ at least $s$ of $\{ \gamma,\gamma_1,\dots,\gamma_n \}$ are in
$G$}}\footnote{This is off by one from the grading convention used in \cite[A1.3.16]{green}.} . $$
There is a short exact sequence of complexes
\begin{equation}\label{filtration-SES} 0\to
F^{s}/\til{F}^{s+1}C_\Gamma^*(M,\Gamma)\to
\til{F}^{s}/\til{F}^{s+1}C_\Gamma^*(M,\Gamma)\to
\til{F}^{s}/F^{s}C_\Gamma^*(M,\Gamma) \to 0.\end{equation}
Unlike $F$, the new filtration
$\til{F}$ preserves the contracting homotopy on $C_\Gamma^*(M,\Gamma)$ given by
$m[\gamma_1|\dots|\gamma_n]\gamma\mapsto
\epsilon(\gamma)m[\gamma_1|\dots|\gamma_{n-1}]\gamma_n$. So
$\til{F}^*C_\Gamma(M,\Gamma)$
is contractible, and so is the quotient complex
$\til{F}^*/\til{F}^{*+1}C_\Gamma(M,\Gamma)$. The short exact sequence
\eqref{filtration-SES} gives rise to a long exact sequence in cohomology, and
contractibility of the middle complex means that the boundary map
\begin{equation}\label{beta-boundary}\delta:H^*(\til{F}^{s}/F^{s}C_\Gamma^*(M,\Gamma))\to
H^*(F^{s}/\til{F}^{s+1}C_\Gamma^{*+1}(M,\Gamma))\end{equation}
is an isomorphism. We will identify $\til{F}^{s}/F^{s}C_\Gamma^*(M,\Gamma)$ and
$F^{s}/\til{F}^{s+1}C_\Gamma^{*+1}(M,\Gamma)$ with the source and target of
the desired map in the lemma statement, and show that $\delta$ can be lifted to
a map on chains.

Levelwise, we can write
\begin{equation}\label{til{F} example}\til{F}^{s+1}C_\Gamma^n(M,\Gamma) =
F^{s+1}C_\Gamma^n(M,\Gamma) + F^{s}C_\Gamma^n(M,G) \end{equation}
but this is an abuse of notation---as $G$ is not a $\Gamma$-comodule,
$C_\Gamma^*(M,G)$ is not a complex (but we can still talk about
$C_\Gamma^n(M,G)\subset C_\Gamma^n(M,\Gamma)$ as a sub-module). 
We will see that this will cease to be a problem upon passing to the associated
graded $E_0$.

For each $n$, we have
\begin{align}\label{F^s-SES-left}
\til{F}^{s}/F^{s}C_\Gamma^*(M,\Gamma) & \isom \big(F^{s}C_\Gamma^n(M,\Gamma) +
F^{s-1}C_\Gamma^n(M,G)\big)\big/F^sC_\Gamma^n(M,\Gamma) 
\\\notag & \isom F^{s-1}/F^sC_\Gamma^n(M,G) 
\label{F^s-SES-right}
\\F^{s}/\til{F}^{s+1}C_\Gamma^{*+1}(M,\Gamma) & \isom
F^{s}C_\Gamma^n(M,\Gamma)\big/\big(F^{s+1}C_\Gamma^n(M,\Gamma)+F^{s}C_\Gamma^n(M,G)\big)
\\\notag &=
\big(F^{s}C_\Gamma^n(M,\Gamma)/F^{s+1}C_\Gamma^n(M,\Gamma)\big)\big/F^{s}C_\Gamma^n(M,G)
\\\notag & \isom F^{s}/F^{s+1}C_\Gamma^n(M,\Sigma).
\end{align}
While $F^sC_\Gamma^*(M,G)$ is not a complex, Lemma \ref{G-cobar-quotient} shows
that $F^{s-1}/F^sC_\Gamma^*(M,G)$ is a complex, and the isomorphisms
$\til{F}^s/F^sC_\Gamma^n(M,\Gamma)\isom F^{s-1}/F^sC_\Gamma^n(M,G)$ and
$F^s/\til{F}^{s+1}C_\Gamma^{*+1}(M,\Gamma)\isom F^s/F^{s+1}C_\Gamma^n(M,\Sigma)$
extend to isomorphisms of complexes.
I claim the boundary map \eqref{beta-boundary} can be identified as the map
\begin{align*}
H^*(F^{s-1}/F^sC^*_\Gamma(M,G)) & \ttoo{\delta}
H^*(F^s/F^{s+1}C^*_\Gamma(\Sigma,N))
\\m[a_1|\dots|a_n]g & \mapstto \sum
m[a_1|\dots|a_n|\underline{\underline{g}}']\underline{\underline{g}}''
\end{align*}
where $\sum \underline{\underline{g}}'|\underline{\underline{g}}''$ is the image
of $g$ under the right $\Sigma$-coaction.
As the boundary map, this is just given by the cobar differential, but in order
for $m[a_1|\dots|a_n]g$ to be a cycle, the sum of all the terms except the one in
the formula for $\delta$ is in $F^sC^{n+1}_\Gamma(M,\Gamma)$. Furthermore, I
claim this can be extended to a map on chains:
\begin{align*}
\delta: F^{s-1}/F^sC^*_\Gamma(M,G) & \tto F^s/F^{s+1}C^*_{\Gamma}(\Sigma,N)
\\m[a_1|\dots|a_n]g & \mapstto
\sum m[a_1|\dots|a_n|\underline{\underline{g}}']\underline{\underline{g}}''.
\end{align*}
It suffices to show that the image of $m[a_1|\dots|a_n]g\in F^s C^*_\Gamma(M,G)$
lies in $F^{s+1}C_\Gamma^*(M,\Sigma)$, and this holds because
$\underline{\underline{g}}''$ is the $(s+1)^{st}$ term in $G$.
\end{proof}

Using Lemma \ref{M-cotensor-Sigma}, we can write this as a map
\begin{align*}
{\let\scriptstyle\textstyle \substack{E_0^{s-1,*}(M,G(s)\cotensor_\Sigma N)\\\ =
E_0^{s-1,*}(M,G)\cotensor_\Sigma N}}\ttoo{\delta}
E_0^{s,*}(M,\Sigma)\cotensor_\Sigma N =
E_0^{s,*} (M,\Sigma\cotensor_\Sigma N)
\ttoo{\isom} E_0^{s,*}(M,N)
\\ m[a_1|\dots|a_n]g|\nu \mapstto
\sum m[a_1|\dots|a_n\underline{\underline{g}}']\underline{\underline{g}}''\nu
\mapstto \sum m[a_1|\dots|a_n|g]\nu.
\end{align*}
\begin{random}{Corollary/ Definition}\label{corollary-definition}
Iterating $\delta$ gives rise to a chain equivalence
$$ \xymatrix{
E_0^{0,*}(M,G(s)\cotensor_\Sigma N)\ttoo{\delta}
E_0^{1,*}(M,G(s-1)\cotensor_\Sigma N) \ttoo{\delta}\dots\ttoo{\delta} E_0^{s,*}(M,N)
}$$
sending $$ m[a_1|\dots|a_n]g_1|\dots|g_s|\nu \mapstto
m[a_1|\dots|a_n|g_1|\dots|g_s]\nu.$$
Let $\beta$ denote this composition.
\end{random}
It is now easy to see that \eqref{CESS-filtration-goal} commutes. Our final task
is to show (4) after \eqref{CESS-filtration-goal}; first we need an easy lemma.


%

\begin{lemma}\label{T}
Let $\Gamma$ be a Hopf algebra and $M$ be an $\Gamma$-comodule.
Then the coaction $\psi:M\to \Gamma\cotensor_\Gamma M$ is an isomorphism with inverse
$T:\Gamma\cotensor_\Gamma M\to M$ sending $a\tensor m\mapsto \epsilon(a)m$.
\end{lemma}
\begin{proof}
First we check that the coaction $\psi$ lands in the cotensor product
$\Gamma\cotensor_\Gamma M$: we need to check that $\psi(m)=\sum m'\tensor m''$ lands in
the kernel of $\Delta\tensor \Id - \Id\tensor \psi:\Gamma\tensor M\to \Gamma\tensor
\Gamma\tensor M$. But $\sum (m')'\tensor (m')''\tensor m'' - \sum m'\tensor
(m'')'\tensor (m'')''=0$ by coassociativity.

Next, we check that $T$ is an inverse. We have $T\psi(m)=T(\sum m'\tensor m'') =
\sum \epsilon(m')m''$. This is equal to $m$ by general Hopf algebra properties.
For the other composition, we have $\psi T(a\tensor m)= \sum
\epsilon(a)m'\tensor m''$. Since $a\tensor m$ is in $\Gamma\cotensor_\Gamma M$,
we have $\sum a\tensor m'\tensor m''=\sum a'\tensor a''\tensor m$. Applying
$\epsilon\cdot \Id\tensor \Id$ to this, we have $\sum\epsilon(a)m'\tensor m'' =
\sum \epsilon(a')a''\tensor m = \sum a\tensor m$. So $\psi\circ T=\Id$.
\end{proof}

\begin{proposition}\label{theta-s=0}
$\theta$ induces an isomorphism $\ED{0,*}1\to E_1^{0,*}$.
\end{proposition}
\begin{proof}
First notice that we have an isomorphism
$$F^0/F^1 (M\tensor \Gamma^{\tensor t}\tensor N) \isom M\tensor \Sigma^{\tensor t}\tensor N $$
since $m[\gamma_1|\dots|\gamma_s]\nu$ is in $F^1$ if any of the $\gamma_i$'s are
in $G$. On the other hand, we have $$H^*(\ED{0,*}1) = H^*((M\tensorD
\Gamma^{\tensorD t+1})\cotensor_\Gamma (\Phi\tensorD N)) =
\Cotor^*_\Gamma(M,\Phi\tensorD N) \isom \Cotor_\Sigma^*(M,N)$$
by the change of rings isomorphism. In the rest of this proof we make this
isomorphism more explicit, enough to see that the isomorphism $\ED{0,*}1\to
E_1^{0,1}$ is induced by $\theta$.

Since the shear map $\Gamma\tensorD \Gamma\to \Gamma\tensorL \Gamma$
commutes with the map $\Gamma\tensor \Gamma\ttoo{q\tensor q} \Sigma\tensor \Sigma$, we
have a commutative diagram
$$ \xymatrix@C=35pt{
(M\tensorD \Gamma^{\tensorD t+1})\cotensor_\Gamma (\Phi\tensorD
N)\ar[d]_-{\Id^{t+2}\tensor S}
\\(M\tensorD \Gamma^{\tensorD t+1})\cotensor_\Gamma
(\Gamma\cotensor_\Sigma N)\ar[r]^{\Id\tensor q^{t+1}\tensor
\Id^2}\ar[d]^{S_c^t\tensor \epsilon\cdot \epsilon\cdot \Id}
& (M\tensorD \Sigma^{\tensorD t+1})\cotensor_\Gamma (\Gamma\cotensor_\Sigma
N)\ar[r]^-\isom & (M\tensorD \Sigma^{\tensorD t+1})\cotensor_\Sigma
N\ar[d]^-{S_c^{t+1}\tensor \epsilon\cdot \Id}
\\F^0/F^1 (M\tensor \Gamma^t \tensor N)\ar[rr]^-\isom && M\tensor \Sigma^t \tensor N
}$$
Note that the left vertical composition is $\theta$, by definition.
The middle horizontal composition is the chain equivalence inducing the change
of rings isomorphism $\Cotor_\Gamma^*(M,\Gamma\cotensor_\Sigma N)\isom
\Cotor_\Sigma(M,N)$. By Lemma \ref{T}, the right vertical map is
$S_c^{t+1}\tensor T$, an isomorphism. So the bottom left vertical map is a chain
equivalence. The top left vertical map is an isomorphism, so $\theta$ is a chain
equivalence.
\end{proof}

\section*{Appendix: The cobar complex and the shear isomorphism}\label{section:cobar}
\def\thesection{A}
\setcounter{theorem}{0}



In this appendix we record some technical facts about the iterated shear
isomorphism that are needed for the comparison proof in section
\ref{section:CESS-vs-filtration}. 

First we need notation for the iterated coproduct.
\begin{definition} \label{iterated-coproduct-notation}
For a Hopf algebra $\Gamma$ and $\Gamma$-comodule $M$, let $\Delta^n$ denote the iterated coproduct
$$\Delta^n:\Gamma\too{\Delta}\Gamma^{\tensor 2}\too{\Delta} \dots
\too{\Delta}\Gamma^{\tensor n+1}$$ and let $\psi^n$ denote the iterated coaction $M\too{\psi^n} \Gamma^{\tensor n}\tensor M$.
Write $\sum m_{(1)}|\dots|m_{(n+1)}\colonequals
\psi^n(m)$ and $\sum \gamma_{(1)}|\dots|\gamma_{(n+1)}\colonequals \Delta^n(\gamma)$.
(Note that this notation is well-defined because of coassociativity.)
\end{definition}
For example, $\Delta(\gamma) = \sum \gamma'| \gamma'' = \sum
\gamma_{(1)}|\gamma_{(2)}$, and $\sum
\Delta(\gamma_{(1)})|\gamma_{(2)} = \sum
\gamma_{(1)}|\gamma_{(2)}|\gamma_{(3)}$.

\begin{lemma}
\label{iterated-shear} The iterated shear isomorphism $S^n:\Gamma^{\tensorD
n}\tensorD M\to \Gamma\tensorL \Gamma^{\tensor n-1}\tensor M$ is given by
$$ S^n: x_1|\dots|x_n|m \mapsto \sum {x_1}_{(1)}{x_2}_{(1)}\dots
{x_n}_{(1)}m_{(1)}|{x_2}_{(2)}\dots {x_n}_{(2)}m_{(2)}|{x_3}_{(3)}\dots
{x_n}_{(3)}m_{(3)}|\dots| m_{(n+1)}.$$
The iterated shear isomorphism $S_c^n:M\tensorD \Gamma^{\tensorD n}\to
M\tensor \Gamma^{\tensor n-1}\tensorR\Gamma$ is given by
$$ S_c^n: m|x_n|\dots|x_1 \mapsto
\sum m_{(1)}|m_{(2)}{x_n}_{(1)}|m_{(3)}{x_n}_{(2)}{x_{n-1}}_{(1)}|\dots|m_{(n+1)}{x_n}_{(n)}{x_{n-1}}_{(n-1)}\dots
{x_2}_{(2)}x_1.$$
\end{lemma}
\begin{proof}
We prove just the first statement, as the second is analogous.
Use induction on $n$. If $n=2$ this is true by definition of $S$. Now suppose
$S^{n-1}$ is given by the formula above. We can write $S^n$ as the composition
$$  \Gamma^{\tensorD n}\tensorD M\ttoo{S^{n-1}} \Gamma\tensorD (\Gamma^{\tensorL
n-1}\tensor M)\ttoo{S} \Gamma\tensorL (\Gamma^{\tensorL n-1}\tensor M) $$
and by the inductive hypothesis the first map sends 
\begin{align*}
x_1|x_2|\dots|x_n|m & \mapsto \sum x_1|{x_2}_{(1)}{x_3}_{(1)}\dots
{x_n}_{(1)}m_{(1)}|{x_3}_{(2)}\dots {x_n}_{(2)}m_{(2)}|\dots|m_{(n)}.
\end{align*}
If we write this as $x_1|y$, then the second map sends this to
$\sum x_1y_{(1)}|y_{(2)}$; remembering that the coaction on $y$ just comes from the
first component, this is:
\begin{align*}
\sum x_1{x_2}_{(1)}{x_3}_{(1)} \dots
{x_n}_{(1)}m_{(1)}|{x_2}_{(2)}{x_3}_{(2)}\dots {x_n}_{(2)}m_{(2)}|{x_3}_{(3)}\dots
{x_n}_{(3)}m_{(3)}|\dots|m_{(n+1)}.  & \qedhere
\end{align*}
\end{proof}

\begin{lemma}\label{iterated-inverse-shear}
The iterated inverse shear isomorphism $S^{-n}:\Gamma\tensorL \Gamma^{\tensor
n-1}\tensor M\to \Gamma^{\tensorD n}\tensorD M$ is given by
$$ S^{-n}: x_1|\dots|x_n|m\mapsto
\sum x_1c(x'_2)|x''_2c(x'_3)|x''_3c(x'_4)|\dots|x''_nc(m')|m''. $$
The iterated inverse shear isomorphism $S_c^{-n}: M\tensor \Gamma^{\tensor
n-1}\tensorR \Gamma \to M\tensorD \Gamma^{\tensorD n}$ is given by
$$ S_c^{-n}:m|x_n|\dots|x_1\mapsto \sum
m'|c(m'')x'_n|c(x'_n)x'_{n-1}|\dots|c(x'_2)x_1. $$
\end{lemma}
\begin{proof}
Again we only prove the first statement, and again this is by induction on $n$.
If $n=1$, this is the definition of $S^{-1}$ in Lemma \ref{shear}.

Assume the formula holds for $n-1$. Write $S^{-n}$ as the composition
$$ \Gamma\tensorL (\Gamma^{\tensorL n-1}\tensor M)\ttoo{S^{-(n-1)}}
\Gamma\tensorL (\Gamma^{\tensorD n}\tensorD M)\ttoo{S^{-1}} \Gamma\tensorD  (\Gamma^{\tensorD n}\tensorD M)$$
and by the inductive hypothesis the first map sends
$$ x_1|x_2|\dots|x_n|m\mapsto \sum
x_1|x_2c(x'_3)|x''_3c(x'_4)|\dots|x''_nc(m')|m''. $$
If we write this as $x_1|y$, then the second map sends this to $\sum
x_1c(y_{(1)})|y_{(2)}$, which is
\begin{align*}
\sum x_1 & c\big((x_2c(x'_3)x''_3c(x'_4) \dots
x''_nc(m')m'')'\big)|x''_2c(x'_3)''|(x''_3)''c(x'_4)''|\dots|(x''_n)''c(m')''|(m'')''
\\ & =\sum x_1c\big({x_2}_{(1)}c({x_3}_{(2)}){x_3}_{(3)}c({x_4}_{(2)}){x_4}_{(3)}\dots
c(m_{(2)})m_{(3)}
\big)|{x_2}_{(2)}c({x_3}_{(1)})|{x_3}_{(4)}c({x_4}_{(1)})|
\\ & \hspace{80pt}\dots|{x_n}_{(4)}c(m_{(1)})|m_{(4)}
\\ & =\sum x_1c\big({x_2}_{(1)}\epsilon({x_3}_{(2)}\dots
{x_n}_{(2)}m_{(2)})\big)|{x_2}_{(2)}c({x_3}_{(1)})|{x_3}_{(3)}c({x_4}_{(1)})|
\\ & \hspace{80pt}\dots|{x_n}_{(3)}c(m_{(1)})|m_{(3)}
\\ & =\sum
x_1c({x_2}_{(1)})|{x_2}_{(2)}c({x_3}_{(1)})|{x_3}_{(2)}c({x_4}_{(1)})|\dots|{x_n}_{(2)}c(m_{(1)})|m_{(2)}.
\end{align*}
Here the first equality uses the fact that $\sum c(x')|c(x'') = \sum
c(x)''|c(x)'$, the second uses the fact that $c(x')x'' = \epsilon(x)$, and the
third uses the fact that $\sum \epsilon(x')|x'' = \sum 1|x$.
\end{proof}

\begin{lemma}
The iterated shear isomorphism $S:\Gamma^{\tensorD *+1}\tensorD N\to
\Gamma^{\tensorL *+1}\tensor N$ restricts to an isomorphism of chain
complexes
\begin{equation}\label{Phi-shear} S:\Phi^{\tensorD *+1}\tensorD N\to
\u{\Gamma\cotensor_\Sigma \dots \cotensor_\Sigma \Gamma}_{*+1}\cotensor_\Sigma N. \end{equation}
\end{lemma}
\begin{proof}
For any $\Gamma$-comodule $M$, by Lemma \ref{shear-cotensor} the shear isomorphism
gives an isomorphism $\Phi\tensorD N \too{\isom} \Gamma\cotensor_\Sigma N$, and
iterating the shear map gives an isomorphism $\Phi^{\tensorD
s+1}\tensorD N\too{\isom} \u{\Gamma\cotensor_\Sigma\dots \cotensor_\Sigma
\Gamma}_{s+1}\cotensor_\Sigma N$.
\end{proof}

\begin{proof}[Proof of Lemma \ref{shear-restriction-G}]
It suffices to check the inclusions $S^{-1}(\Gamma\cotensor_\Sigma G(s)\cotensor_\Sigma
N)\subset \mathcal{N} \DDPhi s(N)$ and $S(\mathcal{N}\DDPhi s(M)) \subset
\Gamma\cotensor_\Sigma G(s)\cotensor_\Sigma N$. For the first inclusion, use
Lemma \ref{iterated-inverse-shear} to observe that
\begin{align}
\label{SG(s)}S^{-1}(a|g_1|\dots|g_s|n) & =
\sum ac(g'_1)|g''_1c(g'_2)|g''_2c(g'_3)|\dots|g''_sc(n')|n''
\end{align}
and for $1\leq i\leq s$ we have
\begin{align*}
\textstyle \mu_i(\sum ac(g'_1)|g''_1c(g'_2)|g''_2c(g'_3)|\dots|g''_sc(n')|n'')
 & =\textstyle \sum
 ac(g'_1)|g''_1c(g'_2)|\dots|g''_{i-1}c(g'_i)g''_ic(g'_{i+1})|\dots|n''
\\ & =\textstyle \sum
 ac(g'_1)|g''_1c(g'_2)|\dots|g''_{i-1}\epsilon(g_i)c(g'_{i+1})|\dots|n''
\end{align*}
which is zero since $g_i\in G$ (and so $g_i\notin k$). This
shows \eqref{SG(s)} is in $\mathcal{N}\DDPhi s(N)$.

For the other direction, let $x_0|\dots|x_s|n \in \mathcal{N}\DDPhi s(N)\subset
\Phi^{\tensorD s+1}\tensor N$. By Lemma \ref{iterated-shear}, we have
\begin{equation}\label{S-x0-formula} S(x_0|\dots|x_s|n) = \sum {x_0}_{(1)}{x_1}_{(1)}\dots
{x_s}_{(1)}n_{(1)}|{x_1}_{(2)}\dots {x_s}_{(2)}n_{(2)} |{x_2}_{(3)}\dots
n_{(3)}|\dots|n_{(s+2)}.\end{equation}
The goal is to show that each component ${x_k}_{(k+1)}{x_{k+1}}_{(k+1)}\dots
{x_s}_{(k+1)}n_{(k+1)}$ is in $G$ for $1\leq k\leq s$. Since $\Phi$ is a left
$\Gamma$-comodule, if $x\in \Phi$ then $\Delta^j(x) = x_{(1)}|\dots|x_{(j)}$ and so $x_{(j)}\in \Phi$. By assumption, all of the $x_i$'s
are in $\Phi$, and since \eqref{S-x0-formula} involves the iterated coproduct
$\Delta^{i+1}(x_i) = {x_i}_{(1)}|\dots|{x_i}_{(i+1)}$ for
every $i$, we have ${x_i}_{(i+1)}\in \Phi$.
If we could guarantee
${x_k}_{(k+1)}$ were in $\bar{\Phi}$, then we would be done (since $G =
\bar{\Phi}\Gamma$). Instead, we show that the terms where ${x_k}_{(k+1)}=1$ sum to
zero.

The terms where ${x_k}_{(k+1)}=1$ are:
\begin{align}\label{x_k=1} 
\sum {x_0}_{(1)}{x_1}_{(1)}\dots {x_{k-1}}_{(1)}{x_k}_{(1)} & \dots
{x_s}_{(1)}n_{(1)} |\dots|{x_{k-2}}_{(k-1)}{x_{k-1}}_{(k-1)}{x_k}_{(k-1)}\dots
\\\notag & |{x_{k-1}}_{(k)}{x_k}_{(k)}{x_{k+1}}_{(k)}\dots
|{x_{k+1}}_{(k+1)}{x_{k+2}}_{(k+1)}\dots
|\dots|n_{(s+2)}.
\end{align}
The assumption that $x_0|\dots|x_s$ is in $\mathcal{N}\DDPhi s(M)$ implies that
$x_{k-1}x_k=0$ (this is where we use the fact that $k\geq 1$), and hence
$$0 = \Delta^k (x_{k-1}x_k) = \sum {x_{k-1}}_{(1)}{x_k}_{(1)} 
|\dots|{x_{k-1}}_{(k-1)}{x_k}_{(k-1)}|{x_{k-1}}_{(k)}{x_k}_{(k)}.$$
Observing how $\Delta(x_{k-1}x_k)$ is embedded in \eqref{x_k=1}, we have
\eqref{x_k=1} $=0$.
\end{proof}


\bibliography{thesis.bib}
\bibliographystyle{alpha}

\end{document}

%% file: usepackage.tex
\usepackage{amssymb}
\usepackage{amsmath}
\usepackage{amsthm}
\usepackage{stmaryrd}
\usepackage{mathrsfs}
\usepackage{color}
\usepackage[colorlinks=true,linkcolor=mblue,citecolor=mblue]{hyperref}
\usepackage{dsfont}
\usepackage[all]{xy}
    \SelectTips{cm}{10}  
\usepackage{float}
\usepackage{calc}
\usepackage{tikz}
\usepackage{mathdots}
\usepackage{colonequals}
\usepackage{adjustbox}
\usepackage{caption}
\usepackage{longtable}
\providecommand{\mparwidth}{1in}
\providecommand{\mtop}{1in}
\providecommand{\mbottom}{1in}
\providecommand{\mleft}{1.2in}
\providecommand{\mright}{1.2in}
\usepackage[top = \mtop, bottom = \mbottom, left = \mleft, right=\mright, marginparwidth=\mparwidth]{geometry}
\usepackage{fancyhdr}
\usepackage{setspace}
\usepackage{mathtools}
\usepackage{scalefnt}
\usepackage{microtype}
\usepackage{pifont}
\usepackage{ifthen}
\usepackage{rotating}
\usetikzlibrary{arrows.meta}

\usepackage{enumitem}
\setlist[enumerate,1]{leftmargin=6ex,topsep=-5pt,label=(\arabic*)}
\setlist[enumerate]{itemsep=7pt}
\setlist[itemize,1]{leftmargin=4ex,topsep=-1em}
\setlist[itemize]{itemsep=5pt}
\setlist[itemize,2]{label=$\circ$}
\setlist[itemize,3]{label={\scalefont{0.6}\color{gray}$\blacktriangleright$}}
\setlist[itemize,4]{label=$\ast$}
\setlist{nolistsep}

%% file: preamble8.tex
\parskip=0.2in \parindent=0in
\allowdisplaybreaks 
\raggedbottom 

\newcommand{\replacecommand}[2]{\providecommand{#1}{}\renewcommand{#1}{#2}}

\renewcommand{\arraystretch}{1.5}

\renewcommand{\l}{\overset}
\newcommand{\into}{\hookrightarrow}
\newcommand{\onto}{\twoheadrightarrow}
\newcommand{\tto}{\longrightarrow}
\newcommand{\too}[1]{\l{#1}\to}
\newcommand{\ttoo}[1]{\l{#1}\longrightarrow}
\newcommand{\intoo}[1]{\l{#1}\into}
\newcommand{\ontoo}[1]{\l{#1}\onto}
\newcommand{\mapstoo}[1]{\l{#1}\mapsto}
\newcommand{\bto}{\leftarrow}
\newcommand{\btto}{\longleftarrow}
\newcommand{\btoo}[1]{\l{#1}\bto}
\newcommand{\bttoo}[1]{\l{#1}\longleftarrow}
\newcommand{\binto}{\hookleftarrow}
\newcommand{\bonto}{\twoheadleftarrow}
\newcommand{\bintoo}[1]{\l{#1}\binto}
\newcommand{\bontoo}[1]{\l{#1}\bonto}
\newcommand{\ointo}{\hspace{3pt}\text{\raisebox{-1.5pt}{$\overset{\circ}{\vphantom{}\smash{\text{\raisebox{1.5pt}{$\into$}}}}$}}\hspace{3pt}}
\newcommand{\lu}{\underset}
\newcommand{\bimplies}{\impliedby}
\newcommand{\ints}{\cap}
\newcommand{\intss}{\bigcap}
\newcommand{\union}{\cup}
\newcommand{\unions}{\bigcup}
\newcommand{\djunion}{\sqcup}
\newcommand{\djunions}{\bigsqcup}
\newcommand{\propersubset}{\subsetneq}
\newcommand{\propersupset}{\supsetneq}
\newcommand{\contains}{\supset}
\newcommand{\semidirect}{\rtimes}
\newcommand{\isom}{\cong}
\newcommand{\normal}{\triangleleft}
\replacecommand{\dsum}{\oplus}
\newcommand{\dsums}{\bigoplus}
\newcommand{\tensor}{\otimes}
\newcommand{\tensors}{\bigotimes}
\newcommand{\cotensor}{{\,\scriptstyle\square}}
\let\originalbar\bar
\renewcommand{\bar}[1]{{\overline{#1}}}
\newcommand{\rlim}{\mathop{\varinjlim}\limits}
\newcommand{\llim}{\mathop{\varprojlim}\limits}
\newcommand{\x}{\times}
\providecommand{\st}{\hspace{2pt} : \hspace{2pt}}
\newcommand{\vv}{\vspace{10pt}}
\newcommand{\til}{\widetilde}
\renewcommand{\hat}{\widehat}
\newcommand{\hhat}{{\hat{\ }}}
\newcommand{\iy}{\infty}
\newcommand{\hteq}{\simeq}
\newcommand{\dd}[2]{\frac{\partial #1}{\partial #2}}
\newcommand{\sm}{\wedge} 
\newcommand{\noqed}{\renewcommand{\qedsymbol}{}}
\newcommand{\adjoint}{\dashv}
\newcommand{\wreath}{\wr}
\newcommand{\heart}{\heartsuit}

\newcommand{\bigast}{\mathop{\vphantom{\sum}\mathchoice%
  {\vcenter{\hbox{\huge *}}}
  {\vcenter{\hbox{\Large *}}}{*}{*}}\displaylimits}

\renewcommand{\dim}{\operatorname{dim}}
\newcommand{\diam}{\operatorname{diam}}
\newcommand{\coker}{\operatorname{coker}}
\newcommand{\im}{\operatorname{im}}
\newcommand{\disc}{\operatorname{disc}}
\newcommand{\Pic}{\operatorname{Pic}}
\newcommand{\Der}{\operatorname{Der}}
\newcommand{\ord}{\operatorname{ord}}
\newcommand{\nil}{\operatorname{nil}}
\newcommand{\rad}{\operatorname{rad}}
\newcommand{\ssum}{\operatorname{sum}}
\newcommand{\codim}{\operatorname{codim}}
\newcommand{\cchar}{\operatorname{char}}
\newcommand{\sspan}{\operatorname{span}}
\newcommand{\rank}{\operatorname{rank}}
\newcommand{\Aut}{\operatorname{Aut}}
\newcommand{\Div}{\operatorname{Div}}
\newcommand{\Gal}{\operatorname{Gal}}
\newcommand{\Hom}{\operatorname{Hom}}
\newcommand{\Mor}{\operatorname{Mor}}
\newcommand{\Vect}{\operatorname{Vect}}
\newcommand{\Fun}{\operatorname{Fun}}
\newcommand{\Iso}{\operatorname{Iso}}
\newcommand{\Map}{\operatorname{Map}}
\newcommand{\Ho}{\operatorname{Ho}}
\newcommand{\Mod}{\operatorname{Mod}}
\newcommand{\Tot}{\operatorname{Tot}}
\newcommand{\cofib}{\operatorname{cofib}}
\newcommand{\fib}{\operatorname{fib}}
\newcommand{\hocofib}{\operatorname{hocofib}}
\newcommand{\hofib}{\operatorname{hofib}}
\newcommand{\Maps}{\operatorname{Maps}}
\newcommand{\Sym}{\operatorname{Sym}}
\newcommand{\Diff}{\operatorname{Diff}}
\newcommand{\Tr}{\operatorname{Tr}}
\newcommand{\Frac}{\operatorname{Frac}}
\renewcommand{\Re}{\operatorname{Re}} 
\renewcommand{\Im}{\operatorname{Im}}
\newcommand{\gr}{\operatorname{gr}}
\newcommand{\tr}{\operatorname{tr}}
\newcommand{\End}{\operatorname{End}}
\newcommand{\Mat}{\operatorname{Mat}}
\newcommand{\Proj}{\operatorname{Proj}}
\newcommand{\Th}{\operatorname{Thom}}
\newcommand{\Thom}{\operatorname{Thom}}
\newcommand{\Spec}{\operatorname{Spec}}
\newcommand{\Ext}{\operatorname{Ext}}
\newcommand{\Cotor}{\operatorname{Cotor}}
\newcommand{\Tor}{\operatorname{Tor}}
\newcommand{\vol}{\operatorname{vol}}

\newcommand{\Set}{\operatorname{Set}}
\newcommand{\Top}{\operatorname{Top}}
\newcommand{\Fin}{\operatorname{Fin}}
\newcommand{\Spaces}{\operatorname{Spaces}}
\newcommand{\Sp}{\operatorname{Sp}}
\newcommand{\Spectra}{\operatorname{Spectra}}
\newcommand{\Spt}{\operatorname{Spt}}
\newcommand{\Comod}{\operatorname{Comod}}
\newcommand{\Spf}{\operatorname{Spf}}
\newcommand{\tmf}{\mathit{tmf}}
\newcommand{\Tmf}{\mathit{Tmf}}
\newcommand{\TMF}{\mathit{TMF}}
\newcommand{\Null}{\operatorname{Null}}
\newcommand{\Fil}{\operatorname{Fil}}
\newcommand{\Sq}{\operatorname{Sq}}
\newcommand{\Stable}{\operatorname{Stable}}
\newcommand{\Poly}{\operatorname{Poly}}
\newcommand{\Cat}{\operatorname{Cat}}
\newcommand{\Orb}{\operatorname{Orb}}
\newcommand{\Exc}{\operatorname{Exc}}
\newcommand{\Part}{\operatorname{Part}}
\newcommand{\Comm}{\operatorname{Comm}}
\newcommand{\Res}{\operatorname{Res}}
\newcommand{\Thick}{\operatorname{Thick}}

\newcommand{\Sm}{\operatorname{Sm}}
\newcommand{\Var}{\operatorname{Var}}
\newcommand{\Frob}{\operatorname{Frob}}
\newcommand{\Rep}{\operatorname{Rep}}
\newcommand{\Ch}{\operatorname{Ch}}
\newcommand{\Shv}{\operatorname{Shv}}
\newcommand{\Corr}{\operatorname{Corr}}
\newcommand{\Span}{\operatorname{Span}}
\newcommand{\Sch}{\operatorname{Sch}}
\newcommand{\ev}{\operatorname{ev}}
\newcommand{\Homog}{\operatorname{Homog}}
\newcommand{\conn}{\operatorname{conn}}
\newcommand{\type}{\operatorname{type}}
\newcommand{\num}{\operatorname{num}}
\newcommand{\Aff}{\operatorname{Aff}}
\newcommand{\Psh}{\operatorname{Psh}}
\newcommand{\sk}{\operatorname{sk}}
\newcommand{\Cart}{\operatorname{Cart}}

\newcommand{\Br}{\operatorname{Br}}
\newcommand{\BW}{\operatorname{BW}}
\newcommand{\Cl}{\operatorname{Cl}}
\newcommand{\Conf}{\operatorname{Conf}}
\newcommand{\Alg}{\operatorname{Alg}}
\newcommand{\CAlg}{\operatorname{CAlg}}
\newcommand{\Lie}{\operatorname{Lie}}
\newcommand{\Coalg}{\operatorname{Coalg}}
\newcommand{\Ab}{\operatorname{Ab}}
\newcommand{\Ind}{\operatorname{Ind}}
\newcommand{\ind}{\operatorname{ind}}
\newcommand{\Fix}{\operatorname{Fix}}
\newcommand{\ho}{\operatorname{ho}}
\newcommand{\coeq}{\operatorname{coeq}}
\newcommand{\CMon}{\operatorname{CMon}}
\newcommand{\Sing}{\operatorname{Sing}}
\newcommand{\Inj}{\operatorname{Inj}}
\newcommand{\StMod}{\operatorname{StMod}}
\newcommand{\Loc}{\operatorname{Loc}}
\newcommand{\Free}{\operatorname{Free}}
\newcommand{\Art}{\operatorname{Art}}
\newcommand{\Gpd}{\operatorname{Gpd}}
\newcommand{\Def}{\operatorname{Def}}
\newcommand{\Hyp}{\operatorname{Hyp}}
\newcommand{\Pre}{\operatorname{Pre}}
\newcommand{\Lat}{\operatorname{Lat}}
\newcommand{\Coords}{\operatorname{Coords}}
\newcommand{\cone}{\operatorname{cone}}
\newcommand{\Spc}{\operatorname{Spc}}
\newcommand{\QCoh}{\operatorname{QCoh}}
\newcommand{\height}{\operatorname{ht}}
\newcommand{\Sub}{\operatorname{Sub}}
\newcommand{\Cone}{\operatorname{Cone}}
\newcommand{\Cocone}{\operatorname{Cocone}}
\newcommand{\Ran}{\operatorname{Ran}}
\newcommand{\Lan}{\operatorname{Lan}}
\newcommand{\LieAlg}{\operatorname{LieAlg}}
\newcommand{\Com}{\operatorname{Com}}
\newcommand{\CoAlg}{\operatorname{CoAlg}}
\newcommand{\Prim}{\operatorname{Prim}}
\newcommand{\Coh}{\operatorname{Coh}}
\newcommand{\FormalGrp}{\operatorname{FormalGrp}}
\newcommand{\Fact}{\operatorname{Fact}}

\newcommand{\A}{\mathbb{A}}
\replacecommand{\C}{\mathbb{C}}
\newcommand{\CP}{\mathbb{C}\mathrm{P}}
\newcommand{\E}{\mathbb{E}}
\newcommand{\F}{\mathbb{F}}
\replacecommand{\G}{\mathbb{G}}
\renewcommand{\H}{\mathbb{H}}
\newcommand{\K}{\mathbb{K}}
\newcommand{\M}{\mathbb{M}}
\newcommand{\N}{\mathbb{N}}
\renewcommand{\O}{\mathcal{O}}
\renewcommand{\P}{\mathbb{P}}
\newcommand{\Q}{\mathbb{Q}}
\newcommand{\R}{\mathbb{R}}
\newcommand{\RP}{\mathbb{R}\mathrm{P}}
\newcommand{\V}{\vee}
\newcommand{\T}{\mathbb{T}}
\providecommand{\U}{\mathscr{U}}
\newcommand{\Z}{\mathbb{Z}}
\renewcommand{\k}{\Bbbk}
\newcommand{\g}{\mathfrak{g}}
\newcommand{\m}{\mathfrak{m}}
\newcommand{\n}{\mathfrak{n}}
\newcommand{\p}{\mathfrak{p}}
\newcommand{\q}{\mathfrak{q}}
\renewcommand{\t}{\mathfrak{t}}

\newcommand{\pa}[1]{\left( {#1} \right)}
\newcommand{\br}[1]{\left[ {#1} \right]}
\newcommand{\cu}[1]{\left\{ {#1} \right\}}
\newcommand{\ab}[1]{\left| {#1} \right|}
\newcommand{\an}[1]{\left\langle {#1}\right\rangle}
\newcommand{\fl}[1]{\left\lfloor {#1}\right\rfloor}
\newcommand{\ceil}[1]{\left\lceil {#1}\right\rceil}
\newcommand{\tf}[1]{{\textstyle{#1}}}
\newcommand{\patf}[1]{\pa{\textstyle{#1}}}

\renewcommand{\mp}{\ \raisebox{5pt}{\text{\rotatebox{180}{$\pm$}}}\ }
\renewcommand{\d}[1]{\ss \mathrm{d}#1}
\newcommand{\imod}{\hspace{-7pt}\pmod}

\renewcommand{\epsilon}{\varepsilon}
\renewcommand{\phi}{{\mathchoice{\raisebox{2pt}{\ensuremath\varphi}}{\raisebox{2pt}{\!\! \ensuremath\varphi}}{\raisebox{1pt}{\scriptsize$\varphi$}}{\varphi}}}
\newcommand{\ph}{{\color{white}.\!}}
\newcommand{\tspacer}{{\ensuremath{\color{white}\Big|\!}}}
\newcommand{\chii}{\raisebox{2pt}{\ensuremath\chi}}

\let\originalchi=\chi
\renewcommand{\chi}{{\!{\mathchoice{\raisebox{2pt}{
$\originalchi$}}{\!\raisebox{2pt}{
$\originalchi$}}{\raisebox{1pt}{\scriptsize$\originalchi$}}{\originalchi}}}}

\let\originalforall=\forall
\renewcommand{\forall}{\ \originalforall}

\let\originalexists=\exists
\renewcommand{\exists}{\ \originalexists}

\let\realcheck\check
\newcommand{\vH}{\realcheck{H}}

\newenvironment{qu}[2]
{\begin{list}{}
	  {\setlength\leftmargin{#1}
	  \setlength\rightmargin{#2}}
	  \item[]\footnotesize}
		  {\end{list}}


\newcommand{\itext}{\shortintertext} 
\renewcommand{\u}{\underbracket[0.7pt]} 
\newcommand{\margin}[1]{\marginpar{\raggedright \scalefont{0.7}#1}} 

\newcommand{\pullback}{\ar@{}[rd]|<<{\text{\pigpenfont A}}}
\newcommand{\pushout}{\ar@{}[rd]|>>{\text{\pigpenfont I}}}

\newcommand{\longleftrightarrows}{\xymatrix@1@C=16pt{
\ar@<0.4ex>[r] & \ar@<0.4ex>[l]
}}
\newcommand{\longrightrightarrows}{\xymatrix@1@C=16pt{
\ar@<0.4ex>[r]\ar@<-0.4ex>[r] & 
}}
\newcommand{\mapstto}{\,\xymatrix@1@C=16pt{
\ar@{|->}[r] & 
}\,}
\newcommand{\mapsttoo}[1]{\xymatrix@1@C=16pt{
\ar@{|->}[r]^-{#1} & 
}}
\newcommand{\rightrightrightarrows}{\xymatrix@1@C=16pt{
\ar[r]\ar@<0.8ex>[r]\ar@<-0.8ex>[r] & 
}}
\newcommand{\longleftleftarrows}{\xymatrix@1@C=16pt{
 & \ar@<0.4ex>[l]\ar@<-0.4ex>[l]
}}
\newcommand{\leftleftleftarrows}{\xymatrix@1@C=16pt{
 & \ar[l]\ar@<0.8ex>[l]\ar@<-0.8ex>[l]
}}
\newcommand{\leftleftleftleftarrows}{\xymatrix@1@C=16pt{
 & \ar@<0.8ex>[l]\ar@<0.3ex>[l]\ar@<-0.3ex>[l]\ar@<-0.8ex>[l]
}}
\newcommand{\lcircle}{\ar@(ul,dl)} 
\newcommand{\rcircle}{\ar@(ur,dr)}
\newcommand{\intto}{\ \xymatrix@1@C=16pt{
\ar@{^(->}[r] & 
}}

\makeatletter
\@ifundefined{mathds}{
	\newcommand{\Id}{Id}
	}{
	\newcommand{\Id}{\mathds{1}} 
	}
\@ifundefined{color}{}{
	\definecolor{darkgreen}{RGB}{0,70,0}
	\definecolor{dgreen}{RGB}{0,100,0}
	\definecolor{purple}{RGB}{120,00,120}
	\definecolor{gray}{RGB}{100,100,100}
	\definecolor{mgreen}{RGB}{0,150,0}
	\definecolor{llgray}{RGB}{230,230,230}
	\definecolor{lgreen}{RGB}{100,200,100}
	\definecolor{mgray}{RGB}{150,150,150}
	\definecolor{lgray}{RGB}{190,190,190}
	\definecolor{maroon}{RGB}{150,0,0}
	\definecolor{lblue}{RGB}{120,170,200}
	\definecolor{mblue}{RGB}{65,105,225}
	\definecolor{dblue}{RGB}{0,56,111}
	\definecolor{orange}{RGB}{255,165,0}
	\definecolor{brown}{RGB}{177,84,15}
	\definecolor{rose}{RGB}{135,0,52}
	\definecolor{gold}{RGB}{177,146,87}
	\definecolor{dred}{RGB}{135,19,19}
	\definecolor{mred}{RGB}{194,28,28}
	\newcommand{\edit}[1]{\itshape{\color{gray}#1}\upshape}
	\newcommand{\fixme}[1]{{\color{maroon}\it{#1}}}
	\newcommand{\citeme}[1]{{\color{orange}\textit{#1}}}
	\newcommand{\later}[1]{{\color{rose}#1}}
	\newcommand{\corr}[1]{{\color{red}\itshape #1}}
	\newcommand{\question}[1]{\itshape{\color{blue}#1}\upshape}
}
\@ifundefined{substack}{}{
    \newcommand{\attop}[1]{{\let\textstyle\scriptstyle\let\scriptstyle\scriptscriptstyle\substack{#1}}}
}
\makeatother


%% file: newtheorem2.tex
\newtheoremstyle{gloss}{\topsep}{\topsep}{}{0pt}{\bfseries}{}{\newline}{\newline
*{\bf #3} }
\theoremstyle{gloss}
\newtheorem*{defstar}{Definition}

\newtheoremstyle{newplain}{20pt}{0pt}{\it}{0pt}{\bfseries}{.}{1ex}{}
\theoremstyle{newplain}

\ifthenelse{\isundefined\theoremnumstyle}
	{\newtheorem{theorem}{Theorem}[section] 
	\numberwithin{equation}{section}} 
	{\ifthenelse{\equal\theoremnumstyle{}}
		{\newtheorem{theorem}{Theorem}}
		{\newtheorem{theorem}{Theorem}[section]
		\numberwithin{equation}{section}
		}
	}

\newtheorem{corollary}[theorem]{Corollary}
\newtheorem{claim}[theorem]{Claim}
\newtheorem{lemma}[theorem]{Lemma}
\newtheorem{proposition}[theorem]{Proposition}
\newtheorem{fact}[theorem]{Fact}

\newtheoremstyle{newtextthm}{20pt}{0pt}{}{0pt}{\bfseries}{.}{1ex}{}
\theoremstyle{newtextthm}
\newtheorem{definition}[theorem]{Definition}
\newtheorem{example}[theorem]{Example}
\newtheorem{problem}[theorem]{Problem}
\newtheorem{remark}[theorem]{Remark}
\newtheorem{notation}[theorem]{Notation}

\newtheorem*{theoremstar}{Theorem}
\newtheorem*{lemmastar}{Lemma}
\newtheorem*{corstar}{Corollary}
\newtheorem*{corollarystar}{Corollary}
\newtheorem*{propositionstar}{Proposition}
\newtheorem*{claimstar}{Claim}
\newtheorem*{examplestar}{Example}

\newcommand{\argforrandom}{}
\theoremstyle{newtextthm}
\newtheorem{helperforrandom}[theorem]{\argforrandom}
\newtheorem*{helperforrandomstar}{\argforrandom}
\newenvironment{random}[1]{\renewcommand{\argforrandom}{#1}\begin{helperforrandom}}{\end{helperforrandom}}
\newenvironment{randomstar}[1]{\renewcommand{\argforrandom}{#1}\begin{helperforrandomstar}}{\end{helperforrandomstar}}

\newenvironment{exercise}[1]{\hspace{1pt}\nn \large {\sc #1.}\hv \normalsize
\vspace{10pt}\\ }{} 
\newcommand{\subthing}[1]{\hv\large(#1)\hv\hv \normalsize }

%% file: newcommands-thesis.tex
\renewcommand{\showlabelfont}{\tiny\color{red}}

\newbox\deltabox
\setbox\deltabox = \hbox{\scalefont{0.2}$\Delta$}
\newcommand{\tensorD}{\l{\copy\deltabox}\tensor}
\newcommand{\cotensorD}{\l{\copy\deltabox}\cotensor}
\newlength\len
\newcommand{\ED}[2]{\setbox3=\hbox{$E^{#1}_{#2}$}\setbox4=\hbox{$E$}\setlength\len{-0.5\wd3+0.5\wd4}\l{\raisebox{-2pt}
{\hspace{1.25\len}\copy\deltabox}}{E^{#1}_{#2}}}

\newbox\lbox
\setbox\lbox = \hbox{\scalefont{0.2}$L$}
\newbox\rbox
\setbox\rbox = \hbox{\scalefont{0.2}$R$}
\newcommand{\tensorL}{\l{\copy\lbox}\tensor}
\newcommand{\tensorR}{\l{\copy\rbox}\tensor}
\newcommand{\DL}[1]{\l{\!\!{\copy\lbox}}{D^{#1}_\Gamma}}
\newcommand{\DR}[1]{\l{\!\!{\copy\rbox}}{D^{#1}_\Gamma}}
\newcommand{\CL}[1]{\l{\!\!{\copy\lbox}}{C^{#1}_\Gamma}}
\newcommand{\CLR}[1]{\l{\!\!{\copy\rbox}}{C^{#1}_\Gamma}}
\newcommand{\CLPhi}[1]{\l{\!\!{\copy\lbox}}{C^{#1}_\Phi}}
\newcommand{\DLPhi}[1]{\l{\!\!{\copy\lbox}}{D^{#1}_\Phi}}
\newcommand{\barDL}[1]{\l{\!\!{\copy\lbox}}{\originalbar{D}^{#1}_\Gamma}}
\newcommand{\DD}[1]{\l{\!\!{\copy\deltabox}}{D^{#1}_\Gamma}}
\newcommand{\CD}[1]{\l{\!\!{\copy\deltabox}}{C^{#1}_\Gamma}}
\newcommand{\DDPhi}[1]{\l{\!\!{\copy\deltabox}}{D^{#1}_\Phi}}
\renewcommand{\cotensor}{\,\text{\scalefont{0.7}$\square$}}
\newcommand{\Imu}{I_{\Phi,\mu}}

\newcommand{\ann}[1]{\langle {#1}\rangle}
\newcommand{\si}{\l{\textit{st}}\isom}
\newcommand{\cc}{\ ;\ }

\newcommand{\EU}{{}^{U\hspace{-3pt}}E}